\documentclass[11pt]{amsart}

\usepackage[T1]{fontenc}
\usepackage[utf8]{inputenc}
\usepackage{lmodern}
\usepackage{amsmath,amssymb,amsthm,mathtools}
\usepackage[margin=0.9in]{geometry}
\usepackage{microtype}
\usepackage{xcolor}
\usepackage{enumitem}
\usepackage[colorlinks=true,linkcolor=blue!55!black,
            citecolor=blue!55!black,urlcolor=blue!55!black]{hyperref}

\newtheorem{theorem}{Theorem}[section]
\newtheorem{proposition}[theorem]{Proposition}
\newtheorem{lemma}[theorem]{Lemma}
\newtheorem{corollary}[theorem]{Corollary}
\theoremstyle{definition}
\newtheorem{definition}[theorem]{Definition}
\theoremstyle{remark}
\newtheorem{remark}[theorem]{Remark}

\DeclareMathOperator{\Ric}{Ric}
\DeclareMathOperator{\Scal}{Scal}
\DeclareMathOperator{\Hess}{Hess}
\DeclareMathOperator{\Vol}{Vol}
\DeclareMathOperator{\diam}{diam}

\newcommand{\R}{\mathbb R}
\newcommand{\eps}{\varepsilon}
\newcommand{\Rm}{\operatorname{Rm}}
\newcommand{\Cint}{\mathcal C}

\newcommand{\Ex}{\mathcal E}

\title{Universal Volume Growth Bounds from Positive\\
Intermediate Curvature}
\author{Gioacchino Antonelli}
\address[Gioacchino Antonelli]{Department of Mathematics, University of Notre Dame, Hurley Hall, 255 Hurley, Notre Dame, IN 46556, United States}
\email{gantonel@nd.edu}
\date{\today}

\begin{document}

\begin{abstract}
Let $n,m$ be integers such that $n\geq2$ and
$0\leq m\leq n-2$. Let $\mathcal C_{m+1}$ denote the
$(m+1)$-intermediate curvature introduced by
Brendle--Hirsch--Johne. 

We prove that there are constants
$\nu(n,m),C(n,m)>0$ such that the following holds. If $(M^n,g)$ is complete and
connected and, for $\delta\geq 0$,
\[
 \Ric\geq-\delta^2,
 \qquad
 \mathcal C_{m+1}\geq 1,
\]
then
\[
 \delta R\leq\nu(n,m)
 \quad\Longrightarrow\quad
 \Vol B_R(p)
 \leq C(n,m)R^m \quad \text{for every $p\in M$ and $R>0$.}
\]

In particular, taking $m=n-2$ and $\delta=0$ gives Gromov's conjectured codimension-two volume
growth estimate under $\Ric \geq0$ and $\Scal \geq1$.
\end{abstract}
\maketitle

\section{Introduction}

In \cite[Definition 1.1]{BrendleHirschJohne}, Brendle--Hirsch--Johne defined and studied notions of intermediate curvatures interpolating between $\mathrm{Ric}$ and $\mathrm{Scal}$ on a Riemannian manifold. 
\begin{definition}[{\cite[Definition 1.1]{BrendleHirschJohne}}]
    Let $(M^n,g)$ be a Riemannian manifold, and let $r$ be an integer such that $1\leq r\leq n-1$. For an $r$-plane $W\subset T_xM$ choose an orthonormal basis $\{e_1,\ldots,e_r\}$ of $W$ and extend it to an orthonormal basis $\{e_1,\ldots,e_r,e_{r+1},\ldots,e_n\}$ of $T_xM$. The \textit{$r$-intermediate curvature} of $W$ is
    \begin{equation}\label{eq:C-definition}
       \Cint_r(W):=\sum_{a=1}^{r}\sum_{b=a+1}^{n}\mathrm{Sect}(e_a\wedge e_b).
\end{equation}
Let $k\in \mathbb R$. When we write $\mathcal{C}_r \geq k$ we mean that $\mathcal{C}_r(W)\geq k$ for every $r$-plane $W\subset T_xM$ and every $x\in M$.
\end{definition}
Notice that $\Cint_1=\Ric$, $\Cint_2=\operatorname{biRic}$ (the biRicci curvature was introduced and studied in \cite{ShenYe}), and
$\Cint_{n-1}=\Scal/2$. The aim of this note is to prove the following volume growth result assuming a Ricci curvature lower bound and positive intermediate curvature. 
\begin{theorem}
\label{thm:main}
Let $n,m$ be integers such that $n\geq2$ and $0\leq m\leq n-2$. There are constants $\nu(n,m)>0$ and $C(n,m)>0$ such that the following holds.

Let $(M^n,g)$ be a complete connected Riemannian manifold, and let $\delta\geq 0$, $\sigma> 0$. Assume
\[
              \Ric\geq -\delta^2,
       \qquad \mathcal{C}_{m+1}\geq\sigma^2.
\]
Then, for every $p\in M$ and every $R>0$,
\begin{equation}\label{eq:main-volume}
       \delta R \leq \nu(n,m) \quad \Longrightarrow \quad \Vol B_R(p)
       \leq C(n,m)\sigma^{-(n-m)}R^m.
\end{equation}
\end{theorem}

Taking $m=n-2$, $\delta=0$, and $\sigma=\frac{\sqrt 2}{2}$ in Theorem \ref{thm:main} we deduce the following.

\begin{corollary}\label{thm:scalar-intro}
Let $n\geq2$, and let $(M^n,g)$ be a complete connected Riemannian manifold
satisfying
\[
                         \Ric\geq0,
              \qquad     \Scal\geq1.
\]
There is a constant $C=C(n)>0$ such that
\[
        \Vol B_R(p)\leq C(n) R^{n-2}
        \qquad\text{for every }p\in M\text{ and every }R>0.
\]
\end{corollary}
Corollary \ref{thm:scalar-intro} answers in the affirmative a question posed by Gromov in \cite[\S2.A(b)]{GromovLarge}. 
\smallskip

Theorem \ref{thm:main} can be used to bound the $m$-dimensional Urysohn width of $(M^n,g)$. Let $m\geq 0$ be an integer. The \textit{$m$-dimensional Urysohn width} of a metric space $X$ is the infimum, over all continuous maps $f\colon X\to P$
to simplicial complexes of dimension at most $m$, of $\sup_{y\in P}\operatorname{diam}\bigl(f^{-1}(y)\bigr)$. Gromov asked whether every complete
$n$-dimensional Riemannian manifold with $\Scal\geq\sigma^2>0$ satisfies $\operatorname{UW}_{n-2}(M)\leq C(n)\sigma^{-1}$;
see, e.g., \cite[\S~2.A(c)]{GromovLarge}.  The
three-dimensional case of this conjecture is known after \cite[Theorem~1.1]{LiokumovichMaximo}, and \cite[Theorem~1.1]{LiokumovichWang}. The analogous macroscopic statement, with positive macroscopic scalar
curvature (see \cite{GuthMacroscopic}) in place of a pointwise scalar-curvature lower bound, is false
in every dimension $n\geq4$, as recently shown by Kumar and Sen
\cite[Theorem~A]{KumarSen}.

Combining Theorem \ref{thm:main} with \cite[Theorem~3.3]{Papasoglu} (see also \cite[Theorem 1.6]{LiokumovichLishakNabutovskyRotman} and the related papers \cite{GuthLargeBalls, GuthUrysohnWidth, NabutovskyWidth}) gives the following
noncollapsed form of the conjectured width-estimate for all the intermediate
curvatures considered here, under an almost non-negative Ricci lower bound. 

\begin{theorem}
\label{cor:noncollapsed-urysohn-width}
Let $n\geq2$, $0\leq m\leq n-2$, and $v_0>0$.  There exist constants
$C=C(n,m,v_0)>0$ and $0<\delta=\delta(n,m,v_0)<1$ with the following property.  If $(M^n,g)$ is a
complete connected Riemannian manifold satisfying
\[
 \Ric\geq -\delta^2,\qquad
 \mathcal{C}_{m+1}\geq1,\qquad
 \inf_{p\in M}\Vol B_1(p)\geq v_0,
\]
then $\operatorname{UW}_m(M,g)\leq C$.
In particular, since $\mathcal{C}_{n-1}=\Scal/2$, we obtain
\[
 \Ric\geq -\delta(n,v_0)^2,\qquad \Scal\geq1,
 \qquad \inf_{p\in M}\Vol B_1(p)\geq v_0
 \quad\Longrightarrow\quad
 \operatorname{UW}_{n-2}(M,g)\leq C(n,v_0).
\]
\end{theorem}
Even when $\mathrm{Ric}\geq 0$, up to the author's knowledge, it is unknown whether the assumption $\inf_{p\in M}\Vol B_1(p)\geq v_0$ in Theorem \ref{cor:noncollapsed-urysohn-width} can be dropped in dimension $n\geq 4$. 
\smallskip

Finally, we mention a consequence of our results for rationally essential manifolds. The aspherical
positive-scalar-curvature conjecture predicts that a closed aspherical
manifold admits no metric of positive scalar curvature (it is solved for dimension $n\leq 5$, see \cite{ChodoshLiAspherical, GromovAspherical5}), while its stronger
rational-essential version replaces asphericity with the
condition \eqref{eqn:RatEs}, see, e.g., \cite[Section 3.2]{GromovFourLectures} (Note that every closed oriented
aspherical manifold is rationally essential).
\begin{definition}
Let $M^n$ be a closed connected oriented manifold, let
$\Gamma=\pi_1(M)$, and let $c_M\colon M\to B\Gamma$ be a classifying
map.  We say that $M$ is \emph{rationally essential} if
\begin{equation}\label{eqn:RatEs}
                     (c_M)_*[M]_{\mathbb Q}\neq 0
                     \quad\text{in }H_n(B\Gamma;\mathbb Q).
\end{equation}
\end{definition}

If $M$ is closed with $\mathrm{Ric}\geq 0$, Cheeger--Gromoll \cite{CheegerGromoll} gives $\widetilde M=N^{n-k}\times\mathbb R^k$, where $N$ is simply connected and compact, $0\leq k\leq n$, and $\pi_1(M)$ contains a finite-index subgroup $\cong \mathbb Z^k$. Thus if $M$ is rationally essential, we must have $k=n$, hence $\widetilde M=\mathbb R^n$,
and thus $M$ is flat.  The point of the following theorem is that, under a dimensionally controlled negative Ricci lower bound, being rationally essential rules out having a positive lower bound on any intermediate curvature; in particular it rules out having a positive lower bound on $\mathrm{Scal}$. The proof of Theorem \ref{prop:rationally-inessential} is a direct consequence of Theorem \ref{thm:main} and \cite[Theorem 1.3]{BraunSauer2021}. For the conclusion on the simplicial volume, compare with \cite[Section 3.A.]{GromovLarge}. For a related result, see \cite[Corollary 2.7]{Lott2025}.
\begin{theorem}\label{prop:rationally-inessential}
Let $n\geq2$ and $0\leq m\leq n-2$.  There exists a constant
$\varepsilon=\varepsilon(n,m)>0$ with the following property.  If
$(M^n,g)$ is closed, connected, and oriented, and 
\[
        \Ric\geq-\varepsilon^2,
        \qquad \mathcal C_{m+1}\geq 1,
\]
then $M$ is not rationally essential.  Hence, $M$ is not
aspherical and its simplicial volume vanishes.
\end{theorem}

\subsection*{History of Corollary \ref{thm:scalar-intro} and previous results}

Recently, there have been many contributions by different authors related to Corollary \ref{thm:scalar-intro}. We list some of them:
\begin{enumerate}[label=\textup{(\arabic*)},leftmargin=2.2em]
\item Under the stronger assumption $\operatorname{Sec}\geq0$, Corollary \ref{thm:scalar-intro} follows from Petrunin's local upper bound for the
scalar-curvature integral \cite{Petrunin};
\item B.~Zhu stated the linear volume growth conclusion of Corollary \ref{thm:scalar-intro} in dimension three under a uniform lower bound for the volumes of unit balls in \cite[Theorem~1.8]{BZhu}.  In $n$ dimensions he stated an $O(R^{n-1})$ estimate under the same noncollapsing assumption and an $O(R^{n-2})$ estimate under a uniform injectivity-radius lower bound \cite[Corollary~1.9]{BZhu};
\item From this item on, in this list, assume $n=3$ and the assumptions of Corollary \ref{thm:scalar-intro}. Munteanu--Wang established the linear-growth estimate with a
      universal constant and also considered scalar-curvature lower
      bounds decaying at infinity \cite{MunteanuWang};
\item Chodosh--Li--Stryker subsequently gave a different proof of the result in \cite{MunteanuWang} and
      extended the power-law decay range, with a constant $C$ depending on
      the manifold and basepoint \cite{ChodoshLiStryker};
\item Wei--Xu--Zhang obtained the sharp upper bound for
      $\limsup_{R\to\infty}\Vol B_R(p)/R$ \cite{WeiXuZhang};
\item Y.~Wang proved an effective local volume upper bound
      \cite{YWang};
\item Huang--Liu obtained sharp asymptotic estimates in the
      asymptotically nonnegative-Ricci setting, under some decay assumptions \cite{HuangLiu}.
\end{enumerate}

In the Ricci-limit setting, Wang--Xie--B.~Zhu--X.~Zhu
\cite{WangXieZhuZhu} excluded an $\R^{n-1}$-splitting for noncollapsed
limits of complete $n$-dimensional Riemannian manifolds with $\Ric\geq0$ and
$\Scal\geq1$; see their Theorem~1.1(1) and the earlier result
\cite{BZhuXZhu}.  In \cite{WangXieZhuZhu} they
also proved
$
       \inf_{p\in M}\Vol B_r(p)\leq c(n)r^{n-2}$ 
under $\Ric\geq0$ and $\Scal\geq1$; see
\cite[Theorem~1.9]{WangXieZhuZhu}.

X.~Zhu obtained related results in \cite{XZhu}.  For example, he proved that asymptotic cones of manifolds with
$\Ric\geq0$, $\Scal\geq1$, and
$\inf_{p\in M}\Vol B_1(p)>0$ have essential dimension at most
$n-2$; see \cite[Theorem 1.6]{XZhu}. Note that Corollary \ref{cor:bounded-splitting} below allows one to remove the noncollapsed
assumption from both \cite[Theorem 1.6]{XZhu} and \cite[Theorem 1.1(1)]{WangXieZhuZhu}.
\vspace{0.1cm}

As
discussed in \cite[Remark~1.5]{WangXieZhuZhu},
for a noncompact $M$ satisfying the assumptions of Corollary \ref{thm:scalar-intro}, we can use Corollary \ref{thm:scalar-intro} to show
$b_1(M)\leq n-3$ (thus removing the noncollapsing hypothesis from the inequality on $b_1$ in \cite[Theorem 1.3]{XZhu}). The estimate on the first Betti number also has a rigidity statement. Let $(M^n,g)$ be a noncompact Riemannian manifold as in the assumptions of Corollary \ref{thm:scalar-intro}.  By \cite{AndersonTopology} and Corollary \ref{thm:scalar-intro} applied to the universal cover of $M$ we get that if $b_1(M)=n-3$, then $M$ has at most linear volume
growth (and thus exactly linear volume growth by Calabi--Yau lower volume bound).  It then follows from
\cite[Theorem~A and Remark~1.2(2)]{HuangHuangSlowGrowth} that
$\widetilde M$ splits off isometrically an
$\mathbb R^{n-3}$-factor. This answers the question raised in
\cite[Remark~5.4]{XZhu} without a uniform noncollapsing assumption. Similar results hold for positive lower bounds on arbitrary intermediate curvatures.
\vspace{0.1cm}

Let us comment further on Theorem \ref{thm:main}.  BiRicci
curvature (corresponding to the case $m=1$ in Theorem \ref{thm:main}) was introduced by Shen--Ye \cite{ShenYe}.
In \cite[Theorem 3(1)]{AntonelliXu} we proved linear volume growth in dimensions $3\leq n\leq 5$ under $\mathrm{Ric}\geq 0$ and a uniformly positive spectral biRicci condition outside a compact set; when the spectral parameter is zero, this includes a pointwise biRicci lower bound.  Zhou--Zhu \cite{ZhouZhu} later extended the result to $3\leq n\leq 7$. Note that Theorem \ref{thm:main} shows that $\mathrm{Ric}\geq 0$ and $\mathrm{biRic}\geq 1$ imply that $M$ has linear volume growth (and then bounded isoperimetric profile), thus answering the global pointwise version of the first bullet of \cite[Question 1]{AntonelliXu} in the negative in every dimension $n\geq 3$. 
\vspace{0.1cm}

{
Corollary \ref{thm:scalar-intro} is closely related to a more general problem posed by Yau \cite{YauProblems}, which asks whether every
complete noncompact $n$-dimensional Riemannian manifold with $\Ric\geq0$ satisfies
\[
       \limsup_{r\to\infty} r^{2-n}
       \int_{B_r(p)}\Scal\,dV <\infty.
\]
Under the additional assumption $\Scal\geq1$, a solution of Yau's conjecture would immediately
imply the volume growth estimate in Corollary \ref{thm:scalar-intro}. Up to the author's knowledge, Yau's conjecture remains open in every dimension $n\geq 3$. For recent progress towards this problem see, e.g., \cite{CucinottaMondino, MunteanuWangIntegral}.
}

\subsection*{Ideas of the proof}

The proof of Theorem \ref{thm:main} has two main ingredients. First, a ball which is quantitatively
close to splitting $m+1$ Euclidean directions has radius at most the
curvature scale; this is Corollary \ref{cor:bounded-splitting}.  Moreover, if a ball is close to
split only $k\leq m$ directions, Proposition \ref{prop:rank-improvement}
finds another ball which is close to split $k+1$ directions while losing only a
harmless multiplicative factor in the $m$-volume ratio. Hence, starting at rank $k=0$
and iterating finitely many times proves the theorem.

The proof of Proposition \ref{prop:rank-improvement} uses tools contained in the proof of
Kapovitch--Wilking rescaling theorem
\cite[Theorem~5.1]{KapovitchWilking}.
We are also inspired by the result in \cite[Lemma~2.12]{Jansen}, which contains a closely related finite formulation of \cite[Theorem 5.1]{KapovitchWilking}; compare with our Proposition \ref{prop:critical-package}. In order to prove Proposition \ref{prop:critical-package} we do not need to use the full strength of \cite[Theorem 5.1]{KapovitchWilking}; see the paragraph before the statement of Proposition \ref{prop:critical-package}. We refer the reader to the beginning of the proof of Proposition \ref{prop:rank-improvement} for an explanation of its proof.
\vspace{0.1cm}

The Hodge
obstruction and the rank-improvement mechanism described above are inspired by
the arguments in \cite{CucinottaMondino, WangXieZhuZhu, XZhu}; in particular, the
Kapovitch--Wilking rescaling theorem
\cite[Theorem~5.1]{KapovitchWilking} is also a
central input in the proof of \cite[Theorem~1.6]{XZhu}.

Moreover, at least two closely related uses of harmonic splitting maps and the Bochner formula precede the argument of Proposition \ref{prop:hodge-bound} and Corollary \ref{cor:bounded-splitting}.  In \cite[Remark~3.1]{WangXieZhuZhu}, the authors give an alternative proof of \cite[Theorem 1.1(1)\&(2)]{WangXieZhuZhu} using harmonic splitting maps with computations similar to the ones in Proposition \ref{prop:hodge-bound}.  On the other hand, in \cite{CucinottaMondino}, the authors use $k$-component splitting maps to obtain integral control of the sum $\mathsf R_k$ of the lowest $k$ Ricci eigenvalues on $k$-symmetric balls; see \cite[Lemma~3.2 and Theorem~3.4]{CucinottaMondino}. Moreover, in \cite[Theorem~3.8(2)]{CucinottaMondino} the authors prove a quantitative obstruction to an $\mathbb R^{n-1}$-splitting under almost nonnegative Ricci curvature, integral almost-nonnegativity of $\operatorname{Ric}_{n-2}$, and integral positivity of the truncated scalar curvature. The key new feature in the proof of Proposition \ref{prop:hodge-bound} is the addition of a contribution coming from the wedge of the one-forms $\beta^a$ to make the intermediate curvature appear: compare with \eqref{eq:mixed-trace}. We further notice that a large part of the proof in this note is likely to work in the setting of metric-measure spaces as well; see Remark \ref{rem:mms}.
\vspace{0.1cm}

For Theorem \ref{cor:noncollapsed-urysohn-width}, noncollapsing and
Bishop--Gromov comparison convert the $O(r^m)$ volume bound into an $O(r^m)$
packing bound for a large annulus.  Averaging over $O(r)$ distance
spheres gives an $O(r^{m-1})$ covering bound for one boundary sphere,
and \cite[Theorem~3.3]{Papasoglu} then gives
bounded $m$-dimensional Urysohn width.

\subsection*{Acknowledgments}

This work was partially supported by NSF grant DMS-2550590. I thank Elia Bruè, Otis Chodosh, Alessandro Cucinotta, Aditya Kumar, Chao Li, Aaron Naber, Daniele Semola,
and Kai Xu for their encouragement, interest in this work, and comments
on preliminary versions.  I am especially grateful to Daniele Semola
for discussions leading to the effective formulation of the results presented here;
to Elia Bruè and Kai Xu for discussions clarifying the
relation of the present argument to the results of Kapovitch--Wilking; to Alessandro
Cucinotta for suggesting the inclusion of
Theorem \ref{cor:noncollapsed-urysohn-width}; and to Aditya Kumar for encouraging me to present the results appeared in the first version of this note for arbitrary Ricci lower bounds and suggesting that I include Theorem \ref{prop:rationally-inessential}.

\subsection*{Disclosure of AI tools.}
This work made substantial use of OpenAI's GPT-5.6 Sol with Ultra reasoning effort and Codex. GPT proposed the central inductive procedure based on the Hodge
obstruction and rank improvement behind the proof of Corollary \ref{thm:scalar-intro}. I formulated and guided the
problem, suggested relevant strategies and literature, and, following the suggested proof strategy of Corollary \ref{thm:scalar-intro}, developed this note.

I emphasize that the proof below is quite different from the original argument suggested by GPT. On the one hand, the proof below is effective, while the original suggested strategy proceeded by contradiction: this emerged through discussions with Daniele Semola, to whom I am very grateful. On the other hand, contrary to the original strategy proposed by the LLM, the proof below is less reliant on the full proof of \cite[Theorem 5.1]{KapovitchWilking}; see the discussion before Proposition \ref{prop:critical-package}: this emerged through discussions with Elia Bruè and Kai Xu. 

After examining the original output, I observed that a modification
of the original argument proposed by the LLM gives the extension to the intermediate curvatures of
Brendle--Hirsch--Johne, as stated in Theorem \ref{thm:main}.  This extension, motivated by a question I
previously had in a joint work with Kai Xu, is my own contribution. Theorems \ref{cor:noncollapsed-urysohn-width} and \ref{prop:rationally-inessential} emerged through discussions with colleagues, ad described in the Acknowledgments.

\subsection*{Addendum}
After this manuscript was completed and circulated privately, I
learned of independent work by Jian Ge \cite{GeGromov}.  Ge proves
Corollary \ref{thm:scalar-intro} by a completely different method, based on
heat-kernel Fisher metric and Nash entropy.  At the same time, I learned of independent work of Bochao Kong and Xingyu Zhu \cite{KZ}. They also establish Corollary \ref{thm:scalar-intro} using heat-kernel techniques. Both approaches are substantially different from the rank-improvement argument used here.
However, I highlight one connection: \cite[Lemma 4.2]{KZ} bears strong similarities with the computations in the proof of Proposition \ref{prop:hodge-bound}.

After completing the revision of this manuscript, in which minor
modifications of the original proofs led to the present form of Theorem
\ref{thm:main}, I learned of independent work of Robert Koirala
\cite[Theorem~1.1]{Koirala2026}. Koirala's result independently implies Theorem~\ref{thm:main} and sharpens it in the following way: it also gives an all-scale estimate with an exponential factor. His proof, based on heat-kernel techniques, is substantially different from the argument used here.

\section{Splitting maps and curvature algebra}

\subsection{Intermediate curvature and volume ratio}\label{sec:Inter}

Let $(M^n,g)$ be a complete Riemannian manifold. We use the convention
\[
 R(X,Y)Z=\nabla_X\nabla_YZ-\nabla_Y\nabla_XZ-\nabla_{[X,Y]}Z,
 \qquad K_{ab}=\mathrm{Sect}(e_a\wedge e_b) = \langle R(e_a,e_b)e_b,e_a\rangle,
\]
where $e_a,e_b$ are orthonormal. In this way we can define Brendle--Hirsch--Johne's intermediate curvatures as in \eqref{eq:C-definition}. Recall that $\Cint_1=\Ric$, $\Cint_2=\operatorname{biRic}$,
$\Cint_{n-1}=\Scal/2$.

For a fixed integer $m\geq 0$ define the \textit{$m$-volume ratio} for $x\in M$ and $r>0$
\begin{equation}\label{eq:excess}
                   \Ex_m(x,r):=\frac{\Vol B_r(x)}{r^m}.
\end{equation}

\subsection{Metric and harmonic splitting}
Whenever a ball occurs as an argument of $d_{GH}$, or as the
domain or target of a Gromov--Hausdorff approximation, it is
understood to be the corresponding closed metric ball. Moreover, we always intend that such
comparisons are pointed with respect to the displayed centers.

\begin{definition}\label{def:metric-split}
Let $0<\eps<1$ and let $k,n$ be integers such that $0\leq k\leq n$.  A ball $B_r(x)$ is
\emph{$(k,\eps)$-split} if there are a pointed proper length space
$(Z,z)$ and a pointed $\eps r$-Gromov--Hausdorff approximation
\[
 B_{\eps^{-1}r}(x)\longrightarrow
 B_{\eps^{-1}r}((0,z))\subset\R^k\times Z.
\]
Note that every ball is $(0,\eps)$-split.
\end{definition}

The following is the quantitative harmonic replacement associated with
Definition \ref{def:metric-split}. It is due to Cheeger--Colding; see
\cite{CheegerColdingWarped,CheegerColdingI} and
\cite[Lemma~1.21(2)]{CheegerNaber}.

\begin{proposition}\label{prop:harmonic-replacement}
For all integers $n,k$ such that $1\leq k\leq n$ and $\eta>0$ there is
$0<\eps=\eps(n,k,\eta)<1$ with the following property. Let $(M^n,g)$ be a complete Riemannian manifold, let $x\in M$ and $r>0$. If
$\Ric\geq -\varepsilon r^{-2}$ and $B_r(x)$ is $(k,\eps)$-split, then there are harmonic
functions
\[
           u=(u^1,\ldots,u^k):B_{4r}(x)\longrightarrow\R^k
\]
such that for every $a=1,\ldots,k$ on $B_{2r}(x)$ we have $ |\nabla u^a|\leq C(n)$ on $B_{2r}(x)$, and moreover
\begin{equation}\label{eq:harmonic-bounds}
 \frac{1}{\Vol B_{2r}(x)}\int_{B_{2r}(x)}
 \left(|G-I|^2+r^2\sum_{a=1}^k|\Hess u^a|^2\right)dV\leq\eta,
\end{equation}
where $G^{ab}:=\langle \nabla u^a,\nabla u^b\rangle$, and $I$ is the identity matrix. Here $|A|$ denotes the Hilbert--Schmidt norm of the matrix $A$.
\end{proposition}

\begin{remark}\label{rem:GHisometry}
Let $\{\varepsilon_i\}_{i\in\mathbb N}$ be a sequence of nonnegative numbers such that $\varepsilon_i\to 0$. If $(M^n_i,g_i)$ are complete Riemannian manifolds with $\mathrm{Ric}_{g_i}\geq -\varepsilon_i$ and
\[
 (M_i,g_i,p_i)\longrightarrow(\mathbb R^k,g_{\mathrm{eu}},0)
\]
in the pointed Gromov--Hausdorff sense, then, up to subsequences, the harmonic functions
$u_i$ in Proposition \ref{prop:harmonic-replacement} can be chosen in the following way:
\begin{enumerate}
    \item $u_i:B_i(p_i)\to B_{i+1}^{\mathbb R^k}(0)$ and $u_i(p_i)=0$;
    \item $u_i$ is a $2^{-i}$-GH approximation from $B_i(p_i)$ to $B_i^{\mathbb R^k}(0)$.
\end{enumerate}  
\end{remark}

\subsection{The Weitzenb\"ock formula}

Let $(M^n,g)$ be a complete $n$-dimensional Riemannian manifold, and let $1\leq p\leq n-1$ be an integer. We shall use the Weitzenb\"ock formula
\begin{equation}\label{eq:weitzenbock}
       \Delta_H=d\delta+\delta d=\nabla^*\nabla+q_p(\Rm)
       \qquad\text{on }\Lambda^pT^*M.
\end{equation}
With our curvature convention,
$q_1(\Rm)=\Ric$.  The only algebraic
fact needed below is the following standard formula.

\begin{lemma}\label{lem:q-diagonal}
Let $\beta^1,\ldots,\beta^n$ be an orthonormal coframe in $(M^n,g)$ and
$I\subset\{1,\ldots,n\}$ with $|I|=p$.  If
$\beta^I=\bigwedge_{a\in I}\beta^a$, then
\begin{equation}\label{eq:q-diagonal}
 \big\langle q_p(\Rm)\beta^I,\beta^I\big\rangle
       =\sum_{a\in I}\sum_{b\notin I}K_{ab}.
\end{equation}
Consequently, the following holds. Let $0\leq m\leq n-2$ be an integer. If $\beta^1,\ldots,\beta^{m+1}$ is an
orthonormal coframe for a $(m+1)$-plane $W$, and
$\Theta=\beta^1\wedge\cdots\wedge\beta^{m+1}$, then
\begin{equation}\label{eq:mixed-trace}
 \big\langle q_{m+1}(\Rm)\Theta,\Theta\big\rangle
 +\sum_{a=1}^{m+1}
   \big\langle q_1(\Rm)\beta^a,\beta^a\big\rangle
       =2\Cint_{m+1}(W).
\end{equation}
\end{lemma}

\begin{proof}
Formula~\eqref{eq:q-diagonal} is in \cite[equation~(11)]{Labbi}.  For completeness, the first
term in \eqref{eq:mixed-trace} is
$\sum_{a\leq m+1<b}K_{ab}$, while the sum of the one-form terms is
\[
  \sum_{a=1}^{m+1}\Ric((\beta^a)^\sharp,(\beta^a)^\sharp)
  =2\sum_{a<b\leq m+1}K_{ab}+\sum_{a\leq m+1<b}K_{ab}.
\]
Their sum is twice \eqref{eq:C-definition}, thus proving \eqref{eq:mixed-trace}, as desired.
\end{proof}

\section{The quantitative Hodge obstruction}

The next proposition is the first core ingredient of our proof.

\begin{proposition}\label{prop:hodge-bound}
Fix an integer $n\geq2$, and an integer $0\leq m\leq n-2$.
There are constants $\eta(n,m)>0$ and $C(n,m)<\infty$ with the
following property. Let $(M^n,g)$ be a complete Riemannian manifold, let $\sigma> 0$, $0\leq \kappa\leq 1$, $x\in M$, and $r>0$. Assume
\[
                 \Ric\geq -\kappa^2r^{-2},
        \qquad   \mathcal{C}_{m+1}\geq\sigma^2,
\]
on $B_{2r}(x)\subset M$, and suppose there are harmonic functions
$u^1,\ldots,u^{m+1}$ on $B_{4r}(x)$ satisfying
\eqref{eq:harmonic-bounds} (and the gradient bound stated before it) for some
$\eta\leq\eta(n,m)$.  Then
\begin{equation}\label{eq:hodge-scale}
                         \sigma^2r^2\leq C(n,m).
\end{equation}
\end{proposition}

\begin{proof}
For $1\leq a,b\leq m+1$, write $\theta^a=du^a$ and
$G^{ab}=\langle\theta^a,\theta^b\rangle$. Let $I$ be the $(m+1)\times (m+1)$-identity matrix. $|A|$ denotes the Hilbert--Schmidt norm of the matrix $A$. Fix a sufficiently small
$\eta_0=\eta_0(n,m)>0$ so that on
\[
                   \Omega:=\{|G-I|<3\eta_0\}
\]
the matrix $G$ is invertible. On $\Omega$ define
\[
 \begin{pmatrix}\beta^1\\ \vdots\\ \beta^{m+1}\end{pmatrix}
       :=G^{-1/2}
         \begin{pmatrix}\theta^1\\ \vdots\\ \theta^{m+1}\end{pmatrix},
 \qquad
                   \Theta:=\beta^1\wedge\cdots\wedge\beta^{m+1}.
\]
Thus $\beta^1,\ldots,\beta^{m+1}$ are orthonormal on $\Omega$.

Choose a smooth function $\chi=\chi(|G-I|^2)$ which equals one when
$|G-I|\leq\eta_0$, vanishes when $|G-I|\geq2\eta_0$, and takes values
in $[0,1]$.  Since
\[
 \nabla G^{ab}
 =\Hess u^a(\,\cdot\,,\nabla u^b)
  +\Hess u^b(\,\cdot\,,\nabla u^a),
\]
the gradient bound stated before \eqref{eq:harmonic-bounds} and differentiation
of $A\mapsto A^{-1/2}$ give
\begin{equation}\label{eq:coframe-derivative}
 |\nabla\chi|^2+|\nabla\Theta|^2
       +\sum_{a=1}^{m+1}|\nabla\beta^a|^2
 \leq C(n,m,\eta_0)\sum_{a=1}^{m+1}|\Hess u^a|^2,
\end{equation}
on $\Omega$.  Note that the forms $\chi\Theta$ and $\chi\beta^a$ extend
smoothly by zero across the boundary of $\Omega$.

Let $f$ be a smooth cutoff function supported in $B_{2r}(x)$, equal to one
on $B_r(x)$, and satisfying $|\nabla f|\leq C(n)/r$.  Set
\[
        \widetilde\omega:=f\chi\Theta,
        \qquad \omega^a:=f\chi\beta^a\quad \forall 1\leq a\leq m+1.
\]
Integrating the identity \eqref{eq:weitzenbock} gives, for any compactly
supported $p$-form $\xi$,
\begin{equation}\label{eq:integrated-W}
 \int\langle q_p(\Rm)\xi,\xi\rangle
 =\int\bigl(|d\xi|^2+|\delta\xi|^2-|\nabla\xi|^2\bigr)
 \leq C(n)\int|\nabla\xi|^2.
\end{equation}
Equation \eqref{eq:mixed-trace} and the
curvature hypothesis give the pointwise lower bound on $B_{2r}(x)$:
\[
 \langle q_{m+1}(\Rm)\widetilde\omega,\widetilde\omega\rangle
 +\sum_{a=1}^{m+1}\langle q_1(\Rm)\omega^a,\omega^a\rangle
 \geq 2\sigma^2 f^2\chi^2.
\]
Summing \eqref{eq:integrated-W} for $\xi$ in the list $(\widetilde\omega,\omega^1,\ldots,\omega^{m+1})$ and using
\eqref{eq:coframe-derivative} and \eqref{eq:harmonic-bounds} therefore yields
\begin{align}
 2\sigma^2\int_{B_{2r}(x)}f^2\chi^2
 &\leq C(n,m,\eta_0)
 \left(r^{-2}\Vol B_{2r}(x)
       +\int_{B_{2r}(x)}\sum_{a=1}^{m+1}|\Hess u^a|^2\right) \notag\\
 &\leq C(n,m,\eta_0)(1+\eta)r^{-2}\Vol B_{2r}(x).
 \label{eq:hodge-upper}
\end{align}

It remains to bound the left-hand side of the above inequality from below.  Markov's
inequality and \eqref{eq:harmonic-bounds} give
\[
 \Vol\bigl(B_{2r}(x)\cap\{|G-I|>\eta_0\}\bigr)
 \leq \eta_0^{-2}\eta\,\Vol B_{2r}(x),
\]
whereas Bishop--Gromov gives
$\Vol B_r(x)\geq c(n)\Vol B_{2r}(x)$.  Choose $\eta(n,m)<1$ so small that 
$\eta(n,m)\leq c(n)\eta_0^2/2$.  Since $f=\chi=1$ on
$B_r(x)\cap\{|G-I|\leq\eta_0\}$, we thus deduce, using the previous inequalities,
\[
        \int f^2\chi^2\geq c(n)\Vol B_{2r}(x)/2.
\]
Combining this with \eqref{eq:hodge-upper} proves
\eqref{eq:hodge-scale}, as desired.
\end{proof}

Combining Propositions \ref{prop:harmonic-replacement} and
\ref{prop:hodge-bound} gives the following.

\begin{corollary}
\label{cor:bounded-splitting}
For every integer $n\geq2$ and integer $0\leq m\leq n-2$ there are
$0<\eps_H(n,m)<1$ and $C(n,m)>0$ such that the following holds. Let $(M^n,g)$ be a complete Riemannian manifold, and fix $x\in M$, $r,\sigma>0$. Let us assume that
\[
         \Ric\geq - \eps_H r^{-2},\qquad \mathcal{C}_{m+1}\geq\sigma^2,
\]
and $B_r(x)$ is $(m+1,\eps_H)$-split. Then
\[
                           r\leq C(n,m)\sigma^{-1}.
\]
\end{corollary}

\section{Critical scales and effective rank improvement}

We record a consequence of Cheeger--Colding's almost splitting theorem \cite{CheegerColdingWarped, CheegerColdingI}. One can give a simple proof by contradiction of Lemma \ref{lem:finite-segment}, based on Gigli's splitting theorem \cite{GigliSplitting}.

\begin{lemma}\label{lem:finite-segment}
For every integer $n\geq 1$, every integer $0\leq k<n$, and $0<\eps<1$, there are
$L=L(n,k,\eps)>\eps^{-1}$ and
$0<\delta_s=\delta_s(n,k,\eps)<1$ with the following property.
Let $(M^n,g)$ be a complete Riemannian manifold, $s>0$, assume $\Ric\geq -\delta_ss^{-2}$, and let $(Z,d_Z,z_0)$ be a pointed proper geodesic space. Assume the closed ball $\overline B_{Ls}(q)$ is
$\delta_s s$-close to the closed $Ls$-ball centered at $(0,z_0)$ in
$\R^k\times Z$.  If $z_0$ is the midpoint of a minimizing 
segment $\gamma:[-Ls,Ls]\to Z$, then
$B_s(q)$ is $(k+1,\eps)$-split.
\end{lemma}

\begin{remark}
    Corollary \ref{cor:bounded-splitting} immediately gives
the following diameter bound. If $0\leq m\leq n-2$, $\{\varepsilon_i\}_{i\in\mathbb N}$ is an infinitesimal sequence, and
\[
 (M_i,g_i,p_i)\longrightarrow(\R^m\times K,(0,z)),
 \qquad
 \Ric_{g_i}\geq -\varepsilon_i,\quad
 \mathcal{C}_{m+1,g_i}\geq\sigma^2>0,
\]
then
\begin{equation}\label{eqn:diamK}
                         \diam K\leq C(n,m)\sigma^{-1}.
\end{equation}
Indeed, recentering at the midpoint of a long segment in $K$ and
applying Lemma \ref{lem:finite-segment} produces an
$(m+1,\eps_H)$-split ball of comparable radius, which is bounded by
Corollary \ref{cor:bounded-splitting}. In the scalar-curvature case this
extends the corresponding diameter results of
\cite{BZhuXZhu} and
\cite[Theorem~1.1]{WangXieZhuZhu}. Note that the present formulation
requires no noncollapsing assumption and applies to all the
intermediate curvatures $\mathcal{C}_{m+1}$. The optimal constant $C(n,m)$ in \eqref{eqn:diamK} is presently unknown. The optimal constant $C$ in Corollary \ref{thm:scalar-intro} and Theorem \ref{thm:main} is also unknown in general.
\end{remark}

We state an elementary lemma. For a Lipschitz $\mathbb R^k$-valued map $b$ defined on a ball $B_2(q)\subset (M^n,g)$, put, for $s\in (0,1)$, and $d$ being the Riemannian distance on $M$,
\begin{equation}\label{eqn:Dqbs}
 D_{q,b}(s):=\sup_{x,y\in\overline B_s(q)}
       \bigl|d(x,y)-|b(x)-b(y)|\bigr|.
\end{equation}

\begin{lemma}\label{lem:distortion-continuity}
The function $s\mapsto D_{q,b}(s)$ is locally Lipschitz on $(0,1)$.
\end{lemma}

\begin{proof}
The function
$(x,y)\mapsto|d(x,y)-|b(x)-b(y)||$ is Lipschitz on $\overline B_s(q)\times \overline B_s(q)$, with Lipschitz constant
controlled by $1+\operatorname{Lip}(b)$.  Since
$d_H(\overline B_s(q),\overline B_{s'}(q))\leq|s-s'|$ for every $s,s'\in (0,1)$, the conclusion
follows.
\end{proof}

The next proposition relies on ideas in the proof of the Kapovitch–-Wilking rescaling theorem \cite[Theorem 5.1]{KapovitchWilking}, and it is inspired by Jansen’s formulation
\cite[Lemma 2.12]{Jansen}. 

We stress that, using the full strength of \cite[Theorem 5.1]{KapovitchWilking} (namely, the fact that the residual factor $K$ in there is unique), one can prove a more general version of Proposition \ref{prop:critical-package} below. We decided to include the weaker statement below for two reasons: first, it has the advantage that its proof (which is inspired by \cite{KapovitchWilking, Jansen}) is essentially self-contained as it only leverages basic facts of Cheeger--Colding theory, \cite[Lemma 2.1]{KapovitchWilking}, and a classical maximal-function argument; second, Proposition \ref{prop:critical-package} is enough to prove our sought Proposition \ref{prop:rank-improvement}.
\begin{proposition}
\label{prop:critical-package}
Fix an integer $n\geq 1$, an integer $0\leq k<n$, and $0<A<\infty$, $0<\eta<1$.
There are $0<\delta_c=\delta_c(n,k,A,\eta)<1$,
$c=c(n)>0$, $C=C(k,n)<\infty$, and $c_0=c_0(n)>0$ with the following
property.  Suppose $(M^n,g,p)$ is complete, has $\Ric_g\geq -\delta_c$, and
\[
 d_{GH}\bigl(B_{\delta_c^{-1}}(p),
             B_{\delta_c^{-1}}^{\R^k}(0)\bigr)\leq\delta_c.
\]
Then there are a measurable set $G\subset B_{1/4}(p)$ and an integer $N\geq 1$ for which the following holds. For every $1\leq j\leq N$ there are $q_j\in B_{1/4}(p)$, $0<t_j\leq 1$, pairwise disjoint measurable $E_j\subset G$, pointed proper geodesic spaces $(Y_j,d_{Y_j},y_j)$, and $a_j\in Y_j$ such that:
\begin{enumerate}[label=\textup{(\roman*)},leftmargin=2.3em]
\item $G=\sqcup_{j=1}^N E_j$ and $\Vol(G)\geq c\Vol B_1(p)$;
\item $\sum_{j=1}^N t_j^k\leq C$, and, for every $1\leq j\leq N$, $E_j\subset B_{Ct_j}(q_j)$;
\item For every $1\leq j\leq N$, $d_{Y_j}(y_j,a_j)\geq c_0t_j$, and
\[
 d_{GH}\bigl(B_{At_j}(q_j),
             B_{At_j}^{\mathbb R^k\times Y_j}((0,y_j))\bigr)\leq\eta t_j.
\]
\end{enumerate}
\end{proposition}

\begin{proof}
Set $c(n):=4^{-n-1}$, $c_*(n):=10^{-n^2}$, $c_0(n):=c_*(n)/4$ and fix $C(k,n)>1$ large enough so that
\begin{equation}\label{eqn:CLarge}
1-C(k,n)^{-1}/2>5c_*/4, \qquad C(k,n)\geq \alpha(k,n)(2\sqrt{k}\mathcal{K}(n)+1)^k,
\end{equation}
where $\mathcal{K}(n)$ is the constant in the gradient bound stated before \eqref{eq:harmonic-bounds}, and $\alpha(k,n)$ is chosen large enough, only depending on $k,n$ (in order for \eqref{eq:sum-variable-scales} below to hold).

The case $k=0$ is elementary. In that case, the assertion is true choosing $\delta_c\ll 1$ and: $N=1$, $G=E_1=B_{1/4}(p)$, $q_1=p$, $t_1=\diam M<\infty$, $(Y_1,d_{Y_1},y_1)=(M,d_g,p)$ and $a_1\in M$ such that $d(a_1,p)\geq \diam M/2$. Hence, from now on, assume $k\geq 1$. 

Suppose that
no $\delta_c>0$ works with the choices of constants above. Taking $\delta_i\to 0$
we thus have a contradicting sequence
\[
        (M_i,g_i,p_i)\longrightarrow(\R^k,g_{\rm eu},0)
\]
with $\mathrm{Ric}_{g_i}\geq -\delta_i$ for which the conclusion is false. 
Proposition
\ref{prop:harmonic-replacement} (see Remark \ref{rem:GHisometry}) and a
diagonal argument give radii $R_i\to\infty$ and harmonic maps
\[
 b_i=(b_i^1,\ldots,b_i^k):B_{R_i}(p_i)\longrightarrow\mathbb R^k,
 \qquad b_i(p_i)=0,
\]
such that the following hold:
\begin{enumerate}[label=\textup{(\alph*)},leftmargin=2.3em]
\item For every $R>0$, the maps $b_i$ have a uniform Lipschitz bound on $B_{R}(p_i)$ for $i$ large enough, and their distance distortion on $B_{R}(p_i)$ tends to
zero as $i\to \infty$;
\item Let  $h_i:=\sum_{a,b=1}^k
 \left|\langle\nabla b_i^a,\nabla b_i^b\rangle-\delta_{ab}\right|^2
       +\sum_{a=1}^k|\Hess b_i^a|^2$. For every fixed $R\geq 2$,
\[
\frac{1}{\Vol B_R(p_i)}\int_{B_R(p_i)}h_i\,dV_{g_i}\longrightarrow0.
\]
\end{enumerate}
Set $
 \alpha_i:=i^{-1}+
 \left(\frac1{\Vol B_2(p_i)}\int_{B_2(p_i)}h_i\,dV_{g_i}\right)^{1/2}>0$, and define
\[
 H_i:=\left\{q\in B_{1/4}(p_i):
       \sup_{0<r\leq3/4}\frac{1}{\Vol B_r(q)}\int_{B_r(q)}h_i\,dV_{g_i}
       \leq\alpha_i\right\}.
\]
Notice $\alpha_i\to 0$. A standard maximal-function estimate and Bishop--Gromov comparison imply
\begin{equation}\label{eq:Hi-large}
 \frac{\Vol(B_{1/4}(p_i)\setminus H_i)}{\Vol B_1(p_i)}\longrightarrow0,
 \qquad
 \Vol(H_i)\geq\frac12\,4^{-n}\Vol B_1(p_i),
\end{equation}
for all sufficiently large $i$.

For $q\in H_i$ define the distortion, for $r\in (0,3/4]$,
\[
 D_{q,b_i}(r):=
 \sup_{x,y\in\overline B_r(q)}
 \left|d_i(x,y)-|b_i(x)-b_i(y)|\right|.
\]
Let $\rho_i(q)$ be the largest real number $r$ in $(0,3/4]$ such that $D_{q,b_i}(r)=c_*r$.
It is readily seen that such an $r$ exists for $i$ large enough. Indeed, since $k<n$, one has $\liminf_{r\to 0}\frac{D_{q,b_i}(r)}r\geq1$. On the other hand, property \textup{(a)} above
gives, uniformly for $q\in H_i$, $
                     D_{q,b_i}(3/4)<3c_*/4$, 
for all sufficiently large $i$.  Lemma \ref{lem:distortion-continuity}
therefore gives the existence of such an $r$. Moreover, notice
$ \sup_{q\in H_i}\rho_i(q)
 \leq c_*^{-1}\operatorname{dis}
       \bigl(b_i|_{\overline B_1(p_i)}\bigr)
 \to 0$ as $i\to \infty$.

We next claim that for every sufficiently large $i$ and every $q\in H_i$ there are a pointed proper geodesic space $(Y,d_Y,y)$ and
$a\in Y$ such that
\begin{equation}\label{eq:individual-product}
 d_Y(y,a)\geq c_0,
 \qquad
 d_{GH}\!\left(
 B_A^{(M_i,\rho_i(q)^{-2}g_i)}(q),
 B_A^{\mathbb R^k\times Y}((0,y))
 \right)\leq\eta.
\end{equation}
Indeed, if this failed, choose violating points $q_i\in H_i$ and put
\[
 \widetilde g_i:=\rho_i(q_i)^{-2}g_i,
 \qquad u_i:=\rho_i(q_i)^{-1}(b_i-b_i(q_i)).
\]
By what was proved above $\rho_i(q_i)\to0$.  Fix $R>0$.
For $i$ large enough, the definition of $H_i$ gives
\begin{align}
 \frac{1}{\Vol^{\widetilde g_i}(B_R^{\widetilde g_i}(q_i))}\left(\int_{B_R^{\widetilde g_i}(q_i)}
 \left(\sum_{a,b=1}^k
 \left|\langle\nabla u_i^a,\nabla u_i^b\rangle-\delta_{ab}\right|^2 + 
 \sum_{a=1}^k|\Hess_{\widetilde g_i}u_i^a|^2\right)
 \,dV_{\widetilde g_i}\right)
 \leq (1+\rho_i(q_i)^2)\alpha_i.\label{eq:scaled-Gram}
\end{align}
Hence, using \eqref{eq:scaled-Gram}, we can apply the result in \cite[Lemma 2.1]{KapovitchWilking} to get, after passing to subsequences,
\begin{equation}\label{eq:critical-product-limit}
 (M_i,\widetilde g_i,q_i,u_i)
 \longrightarrow
 (\mathbb R^k\times Y,(0,y),\operatorname{pr}_{\mathbb R^k}),
\end{equation}
for a pointed proper geodesic space $(Y,y)$, where $\operatorname{pr}_{\mathbb R^k}$ denotes the projection onto the $\R^k$ factor.

Choose $x_i^-,x_i^+\in\overline B_1^{\widetilde g_i}(q_i)$ realizing
the supremum in $D_{q_i,b_i}(\rho_i(q_i))$.  After passing to a subsequence, their
limits are $x^\pm=(v^\pm,y^\pm)$ in
$\mathbb R^k\times Y$.  Convergence in
\eqref{eq:critical-product-limit}, and the fact that $D_{q_i,b_i}(\rho_i(q_i))=c_*\rho_i(q_i)$ give
\[
 \left|
 d_{\mathbb R^k\times Y}(x^-,x^+)-|v^--v^+|
 \right|=d_{\mathbb R^k\times Y}(x^-,x^+)-|v^--v^+|=c_*.
\]
Using the product formula for $d_{\mathbb R^k\times Y}$, we thus deduce $d_Y(y^-,y^+)\geq c_*$. Hence, at least one of $y^-,y^+$ lies at distance at least $c_*/2$ from
$y$.  Since $Y$ is geodesic, there is therefore a point $a\in Y$ with
$d_Y(y,a)=c_0=c_*/4$.  This and
\eqref{eq:critical-product-limit} contradict the assumed failure of
\eqref{eq:individual-product}.  Thus \eqref{eq:individual-product}
holds uniformly over $H_i$, as desired.

We now construct the $E_j$'s.  Apply Vitali's covering Lemma to $
 \left\{B_{\rho_i(q)/10}^{\mathbb R^k}(b_i(q)):q\in H_i\right\}$. 
We thus have a countable pairwise disjoint subfamily,
indexed by points $q_{i,j}\in H_i$, whose  five-times-enlargements cover
$b_i(H_i)$. Since the functions $b_i$ are $\sqrt{k}\mathcal{K}(n)$-Lipschitz on $B_1(p_i)$, all the pairwise disjoint
balls $B_{\rho_i(q_{i,j})/10}(b_i(q_{i,j}))$ lie in
$B_{2\sqrt{k}\mathcal{K}(n)+1}^{\mathbb R^k}(0)$.  Euclidean volume comparison gives (recall \eqref{eqn:CLarge}), for a constant $\beta(k,n)$,
\begin{equation}\label{eq:sum-variable-scales}
 \sum_j\rho_i(q_{i,j})^k
 \leq \beta(k,n) \mathcal L^k\bigl(B_{2\sqrt{k}\mathcal{K}(n)+1}^{\mathbb R^k}(0)\bigr)\leq C(n,k).
\end{equation}

The measurable sets $
 F_{i,j}:=H_i\cap b_i^{-1}
 \left(B_{\rho_i(q_{i,j})/2}^{\mathbb R^k}(b_i(q_{i,j}))\right)$
cover $H_i$.  Define $E_{i,1}:=F_{i,1}$ and $E_{i,j}:=F_{i,j}\setminus \cup_{\ell<j}F_{i,\ell}$ for $j>1$. Notice that $\{E_{i,j}\}_j$ are pairwise disjoint. We claim that
\begin{equation}\label{eq:anchored-cell}
                              E_{i,j}\subset B_{C\rho_i(q_{i,j})}(q_{i,j}).
\end{equation}
If $x\in E_{i,j}\subset F_{i,j}$ and $D=d(x,q_{i,j})\geq C\rho_i(q_{i,j})$, then $D\leq1/2$ (recall $H_i\subset B_{1/4}(p_i)$) and
\eqref{eqn:CLarge} gives
\[
 D_{q_{i,j},b_i}(5D/4)
 \geq D-|b_i(x)-b_i(q_{i,j})|
 \geq D-\rho_i(q_{i,j})/2
 >c_*\,5D/4.
\]
Since $5D/4<3/4$, the latter inequality and
$D_{q_{i,j},b_i}(3/4)<3c_*/4$ (for $i$ large enough) together with
Lemma \ref{lem:distortion-continuity}, gives a 
contradiction to the definition of $\rho_i(q_{i,j})$ and proves
\eqref{eq:anchored-cell}. Indeed, by the previous inequalities and Lemma \ref{lem:distortion-continuity}, there is $r\in (5D/4,3/4)$ such that $D_{q_{i,j},b_i}(r)=c_*r$. Consequently, by definition of $\rho_{i}(q_{i,j})$, we have $\rho_i(q_{i,j})\geq r>5D/4>D\geq C\rho_{i}(q_{i,j})>\rho_{i}(q_{i,j})$, a contradiction.

Now, let us retain
finitely many $\{E_{i,j}\}_{j=1}^{N_i}$ and set $G_i:=\sqcup_{j=1}^{N_i} E_{i,j}$ so that $\Vol(G_i)\geq \Vol(H_i)/2$. For each $q_{i,j}$, denote $\widehat Y_{i,j}:=(Y_{i,j},\rho_i(q_{i,j})d_{Y_{i,j}},y_{i,j})$ where $Y_{i,j}$ is the pointed metric space in \eqref{eq:individual-product} associated with $q_{i,j}$; retain also the associated points $a_{i,j}\in \widehat Y_{i,j}$. Now, for one fixed sufficiently large $i$, equations
\eqref{eq:Hi-large}, \eqref{eq:individual-product},
\eqref{eq:sum-variable-scales}, and \eqref{eq:anchored-cell} give
all three conclusions of the proposition with $G=G_i$, $N=N_i$, $q_j=q_{i,j}$, $t_{j}=\rho_i(q_{i,j})$, $E_j=E_{i,j}$, $a_{i,j}\in \widehat Y_{i,j}$ defined above, thus resulting in a contradiction.
\end{proof}

We now prove the second core ingredient of the proof. 

\begin{proposition}
\label{prop:rank-improvement}
Fix an integer $n\geq2$, integers $m,k$ such that $0\leq m\leq n-2$, $0\leq k\leq m$, and fix
$0<\eps<1$.  There are
$0<\delta=\delta(n,m,k,\eps)<1$ and
$c=c(n,m,k,\eps)>0$ such that the following holds. Let $(M^n,g)$ be a complete Riemannian manifold, $p\in M$ and $r>0$. If
$\Ric\geq -\delta r^{-2}$ on $M$ and $B_r(p)$ is $(k,\delta)$-split, then there are
$q\in M$ and $0<s\leq r$ such that $B_s(q)$ is
$(k+1,\eps)$-split and
\begin{equation}\label{eq:rank-volume}
                        \Ex_m(q,s)\geq c\Ex_m(p,r).
\end{equation}
\end{proposition}

\begin{proof}
All constants denoted by $c,C,C'>0$ in this proof depend only on
$n,m,k,\varepsilon$ and may change from line to line. It is enough to prove the rescaled version with
$r=1$. So, from now on, let us assume $r=1$.
Notice
\[
                         \Ex_m(p,1)=\Vol B_1(p).
\]
Assume $B_{1}(p)$ is $(k,\delta)$-split. Let
\begin{equation}\label{eqn:Original}
 \Phi:B_{\delta^{-1}}(p)\longrightarrow
 B_{\delta^{-1}}((0,z))\subset\mathbb R^k\times Z
\end{equation}
be the $\delta$-approximation given by  Definition \ref{def:metric-split}.
\smallskip

The proof is divided into two cases. If the residual factor $Z$ has a long enough segment through $z$, Lemma \ref{lem:finite-segment} directly produces an additional Euclidean direction. Otherwise Proposition
\ref{prop:critical-package} gives a collection of
good cells $E_j$ in $M$ at different scales, with bounded diameter.  Since their scales are summable, we can select one cell carrying a
definite fraction of the $m$-volume ratio. A segment in the space $Y_j$ associated to such a cell supplies the additional Euclidean direction. Let us carefully perform this strategy.
\smallskip

Let $L(n,k,\varepsilon)>\varepsilon^{-1}>1,0<\delta_s(n,k,\varepsilon)<1$ be the constants in Lemma \ref{lem:finite-segment}, and let $c_0$ be the constant in Proposition \ref{prop:critical-package}. Choose
\[
 0<\vartheta\le c_0/(32L),\quad
 0<\eta\le\min\{1/2,\delta_s\vartheta/C_1\},\quad
 A>c_0/4+L\vartheta+2,\quad
 d_0=\delta_c/20,\quad R_0=4\delta_c^{-1},
\]
where $C_{1}$ is a large universal constant that controls restriction and recentering of a
Gromov--Hausdorff approximation, and $\delta_c(n,k,A,\eta)<1$ is the corresponding constant in
Proposition \ref{prop:critical-package}. Finally choose $\delta$ sufficiently
small so that $0<\delta\leq \min\{1,\delta_s,\delta_c\}$ and:
\begin{equation}\label{eqn:ChoiceOfDelta}
    1+\frac{3d_0}{4}\leq \delta^{-1}, \qquad C_1\delta\leq \frac{\delta_s d_0}{4L}, \qquad \delta_c^{-1}<\frac{\delta^{-1}+1}{10}, \quad C_1\delta\leq \frac{\delta_c}{2}.
\end{equation}

\smallskip
\noindent
\emph{Case 1: There is a long enough segment through $z$ in $Z$.}
Suppose that some $z'\in B_{R_0}(z)$ satisfies
$d_Z(z,z')\geq d_0$. Let $\gamma:[0,d_0]\to Z$ be the unit-speed initial segment of a minimizing geodesic from $z$ to $z'$. Set
\[
 \bar z:=\gamma(d_0/2),
 \qquad
 s_0:=\frac{d_0}{4L}.
\]
Since $Ls_0=d_0/4$, the restriction of $\gamma$ to
$[d_0/4,3d_0/4]$, reparametrized on $[-Ls_0,Ls_0]$, is a minimizing
segment centered at $\bar z$. Choose $q\in M$ corresponding to $(0,\bar z)$ under the approximation \eqref{eqn:Original}. First
$ \overline B_{Ls_0}((0,\bar z))
 \subset \overline B_{3d_0/4}((0,z))$,
so this ball lies inside the original comparison region since $1+3d_0/4\le \delta^{-1}$ (see \eqref{eqn:ChoiceOfDelta}). Restriction
and recentering give
\[
 d_{GH}\!\left(
 \overline B_{Ls_0}(q),
 \overline B_{Ls_0}((0,\bar z))
 \right)
 \leq C_1\delta,
\]
for a universal constant $C_1>0$.
Since
 $C_{1}\delta\leq\delta_s s_0$ (see \eqref{eqn:ChoiceOfDelta}), and $\mathrm{Ric}\geq -\delta\geq -\delta_ss_0^{-2}$ (using $s_0\leq 1$) all the hypotheses of Lemma \ref{lem:finite-segment} are
satisfied, and consequently $B_{s_0}(q)$ is
$(k+1,\varepsilon)$-split.

Since $p$ and $q$ correspond respectively to $(0,z)$ and
$(0,\bar z)$ in the approximation given by \eqref{eqn:Original},
\[
 d(p,q)\leq d_Z(z,\bar z)+C_{\rm 1}\delta
          =\frac{d_0}{2}+C_{\rm 1}\delta
          \leq C(n,m,k,\varepsilon),
\] 
and hence $B_1(p)\subset B_{C}(q)$ for a possibly larger constant $C>\max\{1,s_0\}$.
Bishop--Gromov monotonicity then gives, with $c:=c(n,m,k,\varepsilon)$ possibly changing from line to line
\[
 \begin{split}
 \Ex_m(q,s_0)
 &=\frac{\Vol B_{s_0}(q)}{s_0^m}
 \geq \frac{c}{s_0^m}\Vol B_C(q)
 \geq c\,\Vol B_1(p)
 =c\Ex_m(p,1).
 \end{split}
\]
We thus proved that $B_{s_0}(q)$ is $(k+1,\varepsilon)$-split and $\Ex_m(q,s_0)\geq c\Ex_m(p,1)$. Hence, the proof is concluded in the first case. Notice that we have $s_0=d_0/4L\leq 1=r$, as desired.

\smallskip
\noindent
\emph{Case 2: The residual factor $Z$ is macroscopically small.}
Suppose now that every point of $B_{R_0}(z)$ lies within $d_0$ of
$z$. Since $\delta_c^{-1}<R_0$, projection onto the Euclidean factor
gives
\[
 d_{GH}\!\left(
 B_{\delta_c^{-1}}^{\mathbb R^k\times Z}((0,z)),
 B_{\delta_c^{-1}}^{\mathbb R^k}(0)\right)\leq2d_0.
\]
Hence, using \eqref{eqn:Original} and the third inequality in \eqref{eqn:ChoiceOfDelta}
\[
 d_{GH}\!\left(
 B_{\delta_c^{-1}}(p),
 B_{\delta_c^{-1}}^{\mathbb R^k}(0)\right)\leq\delta_c,
\]
where we also used $C_1\delta\leq \delta_c/2$, see \eqref{eqn:ChoiceOfDelta}. Moreover $\mathrm{Ric}\geq - \delta\geq -\delta_c$. 
Apply Proposition \ref{prop:critical-package}, and denote its output
by $G$ and $\{q_j,t_j,E_j,Y_j,y_j,a_j\}_{j=1}^N$.
Since $k\leq m$ and $0<t_j\leq1$, we have $
                    \sum_{j=1}^N t_j^m
                    \leq\sum_{j=1}^N t_j^k\leq C$. 
Consequently, for some $j$
\[
 \frac{\Vol(E_j)}{t_j^m}
 \geq\frac{\Vol(G)}{\sum_{\ell=1}^N t_\ell^m}
 \geq c\Vol B_1(p).
\]
Since $E_j\subset B_{Ct_j}(q_j)$, we have the following inequality where the constant $c$ might change from line to line
\begin{equation}\label{eq:heavy-cell}
 \Ex_m(q_j,Ct_j)
 =\frac{\Vol B_{Ct_j}(q_j)}{(Ct_j)^m}
 \geq c\frac{\Vol(E_j)}{t_j^m}
 \geq c\Vol B_1(p)=c\Ex_m(p,1).
\end{equation}

Let $\gamma:[0,c_0t_j/2]\to Y_j$ be the initial part of a minimizing
geodesic from $y_j$ to $a_j$, and put $
                       \bar y:=\gamma(c_0t_j/4)$, and $s:=\vartheta t_j$.
Thus $\bar y$ is the midpoint of a segment of half-length $c_0t_j/4$.
Since $Ls=L\vartheta t_j\leq c_0t_j/32$, the central subsegment of
half-length $Ls$ satisfies the segment hypothesis of
Lemma \ref{lem:finite-segment}.  Moreover, by the choice of $A$
\[
 \overline B_{Ls}((0,\bar y))
 \subset
 \overline B_{(c_0/4+L\vartheta)t_j}((0,y_j))
 \subset \overline{B}_{At_j}((0,y_j)).
\]
Choose $q\in M$ corresponding to $(0,\bar y)$ in the product
approximation given by Proposition \ref{prop:critical-package} at $q_j$.  Restriction and recentering give
\[
 d_{GH}\!\left(
 \overline B_{Ls}(q),
 \overline B_{Ls}((0,\bar y))
 \right)
 \leq C_1\eta t_j
 \leq\delta_s\vartheta t_j=\delta_ss.
\]
Since $\mathrm{Ric}\geq - \delta\geq -\delta_ss^{-2}$ (using $s=\vartheta t_j\leq 1$), Lemma \ref{lem:finite-segment} therefore shows that
$B_s(q)=B_{\vartheta t_j}(q)$ is $(k+1,\varepsilon)$-split.

The proof is now concluded as in Case 1: we have $d(q,q_j)\leq Ct_j$, so
$B_{Ct_j}(q_j)\subset B_{C't_j}(q)$. Bishop--Gromov monotonicity and
\eqref{eq:heavy-cell} yield, possibly by changing constants $c,C$ from line to line,
\[
 \begin{split}
 \Ex_m(q,\vartheta t_j)
 &=
   \frac{\Vol B_{\vartheta t_j}(q)}{(\vartheta t_j)^m}\\
 &\geq
   \frac{c}{(\vartheta t_j)^m}
   \Vol B_{C't_j}(q)\geq c\frac{\Vol B_{Ct_j}(q_j)}{t_j^m}
 \geq c\Ex_m(q_j,Ct_j)
 \geq c\Ex_m(p,1),
 \end{split}
\]
thus concluding the proof. Notice that we also have $\vartheta t_j\leq 1=r$, as desired.
\end{proof}
\begin{remark}\label{rem:mms}
    Proposition \ref{prop:rank-improvement} is likely to work, with minimal changes, in the setting of metric-measure spaces $(X,\mathrm{d},\mathcal{H}^n)$ that are $\mathrm{RCD}(-\delta r^{-2},n)$. In particular, in the proof of Theorem \ref{thm:main}, the smoothness of $M$ and the lower bound on the intermediate curvature enters into play only in the computations of Proposition \ref{prop:hodge-bound}, especially when one uses the equality \eqref{eq:mixed-trace}. It would be interesting to understand how (and to which extent) these computations can be adapted in the metric-measure space setting.  We do not pursue this line of research here.
\end{remark}

\section{Proof of the results}

\begin{proof}[Proof of Theorem \ref{thm:main}]
Let $\eps_{m+1}:=\eps_H(n,m)$ be the constant in
Corollary~\ref{cor:bounded-splitting}. For
$k=m,m-1,\ldots,0$, choose recursively
\[
 \eps_k:=\delta(n,m,k,\eps_{k+1}),
 \qquad
 c_k:=c(n,m,k,\eps_{k+1}),
\]
where $\delta$ and $c$ are the constants in
Proposition~\ref{prop:rank-improvement}. Set
\[
 \nu(n,m):=\min\{1,\sqrt{\eps_0},\ldots,\sqrt{\eps_{m+1}}\}>0.
\]

Fix a ball $B_R(p)\subset M$ such that $\delta R\leq\nu(n,m)$ (Here, $\delta$ is the one in the statement of Theorem \ref{thm:main}).
Recall that the output scale $s$ in
Proposition~\ref{prop:rank-improvement} is such that $0<s\leq r$. Since $B_R(p)$ is
$(0,\eps_0)$-split, successive applications of Proposition~\ref{prop:rank-improvement} give
balls $B_{r_k}(p_k)$, $0\leq k\leq m+1$, such that
\[
 r_0=R,\qquad p_0=p,
 \qquad
 R=r_0\geq r_1\geq\cdots\geq r_{m+1}>0,
\]
$B_{r_k}(p_k)$ is $(k,\eps_k)$-split, and
\begin{equation}\label{eq:induction-excess}
 \Ex_m(p_{k+1},r_{k+1})
 \geq c_k\Ex_m(p_k,r_k)
 \qquad \forall 0\leq k\leq m.
\end{equation}
Indeed, at the $k$th step, the choice of $\nu(n,m)$ and $r_k\leq R$ give $
 \delta^2r_k^2\leq\delta^2R^2
 \leq\nu(n,m)^2\leq\eps_k$ 
and hence $
 \Ric_g\geq-\delta^2g\geq-\eps_k r_k^{-2}g$
which is the Ricci hypothesis required in
Proposition~\ref{prop:rank-improvement}.

The final ball is $(m+1,\eps_H)$-split and, again by the choice of
$\nu(n,m)$, we have  $\delta^2r_{m+1}^2\leq\eps_{m+1}=\eps_H(n,m)$.
Therefore Corollary~\ref{cor:bounded-splitting} gives $
 r_{m+1}\leq C(n,m)\sigma^{-1}$. 
Moreover, Bishop volume comparison and
$\delta r_{m+1}\leq\delta R\leq\nu(n,m)\leq 1$ give $
 \Vol B_{r_{m+1}}(p_{m+1})
 \leq C(n,m)r_{m+1}^n$.
Consequently,
\[
 \Ex_m(p_{m+1},r_{m+1})
 \leq C(n,m)r_{m+1}^{n-m}
 \leq C(n,m)\sigma^{-(n-m)}.
\]
Iterating \eqref{eq:induction-excess} yields the sought conclusion
\[
 \frac{\Vol B_R(p)}{R^m}
 =\Ex_m(p_0,r_0)
 \leq\left(\prod_{k=0}^{m}c_k^{-1}\right)
       \Ex_m(p_{m+1},r_{m+1})
 \leq C(n,m)\sigma^{-(n-m)}. \qedhere
\]
\end{proof}

\begin{remark}
    It is likely that the techniques developed in this note can be used to show versions of Theorem \ref{thm:main} where the positive intermediate curvature condition $\mathcal{C}_{m+1}\geq \sigma^2$ is weakened to a lower bound that goes to $0$ at infinity like a (properly chosen) negative power of the distance from a point (cf., the results in \cite{MunteanuWang, ChodoshLiStryker}). Since this is out of the scope of this note, we do not pursue it here.
\end{remark}

\begin{proof}[Proof of Theorem \ref{cor:noncollapsed-urysohn-width}]
The case $m=0$ is a consequence of Bonnet--Myers theorem. We may
therefore assume $1\leq m\leq n-2$.

We will be choosing $\delta$ small enough during the proof. By Theorem \ref{thm:main}, there is $A=A(n,m)$ such that for every $p\in M$, if $\delta R\leq \nu(n,m)$, then
\begin{equation}\label{eq:width-volume-input}
                 \Vol_g B_R(p)\leq A R^m.
\end{equation}
Fix $x\in M$ and $r\geq2$ a large enough integer, only depending on $n,m,v_0$ (see below).  In the
closed annulus $\overline B_{2r}(x)\setminus B_r(x)$ choose a maximal
family of points $p_1,\ldots,p_N$ for which the open balls
$B_{1/10}(p_i)$ are pairwise disjoint.  Bishop--Gromov comparison and
the noncollapsing assumption give
\[
 \Vol_g B_{1/10}(p_i)
 \geq c(n)\Vol_g B_1(p_i)
 \geq c(n)v_0.
\]
All these balls lie in $B_{2r+1/10}(x)$; hence, if $\delta$ is chosen sufficiently small, $\delta(2r+1/10)\leq \nu(n,m)$ and
\eqref{eq:width-volume-input}, together with the previous inequality, implies $
                         N\leq C_0(n,m,v_0)r^m$. 

By maximality, the balls $B_{1/5}(p_i)$ cover the annulus.  Consider
the distance spheres $\partial B_j(x)$ for $ j\in \mathbb N\cap(r,2r)$.
Each open $1/5$-ball meets at most one of these unit-spaced spheres.
Since $N\leq C_0 r^m$, there is $j\in(r,2r)$ for which $\partial B_j(x)$ is
covered by at most $C_1(n,m,v_0)r^{m-1}$ balls of radius $1/5$.  Denoting by $\mathrm{HC}$ the Hausdorff content, we thus have
\begin{equation}\label{eq:width-boundary-content}
       \operatorname{HC}_m\bigl(\partial B_j(x)\bigr)
       \leq C_2(n,m,v_0)r^{m-1}.
\end{equation}

Let $\varepsilon_{m+1}>0$ be the constant in
\cite[Theorem~3.3]{Papasoglu}. The choice of $r=r(n,m,v_0)\geq2$ above should be made in such a way
that the right-hand side of \eqref{eq:width-boundary-content} is $\leq \varepsilon_{m+1}r^m$ (and then, accordingly, $\delta$ must be chosen small with respect to this choice of $r$).  For each $x$, set $U_x=B_j(x)$, choosing
$j=j(x)$ as above.  Then
\[
 B_r(x)\subset U_x\subset B_{2r}(x)\subset B_{10r}(x),
 \qquad
 \operatorname{HC}_m(\partial U_x)\leq
 \varepsilon_{m+1}r^m.
\]
Hence \cite[Theorem 3.3]{Papasoglu} applied with parameter $m+1$ gives
$\operatorname{UW}_m(M,g)\leq r$, as desired.
\end{proof}

\begin{proof}[Proof of {Theorem \ref{prop:rationally-inessential}}]
Suppose that $M$ is rationally essential.  By
\cite[Theorem~1.3]{BraunSauer2021}, there is $c(n)>0$ such that the
Riemannian universal cover $(\widetilde M,\widetilde g)$ satisfies
\[
 \sup_{\widetilde p\in\widetilde M}
 \Vol_{\widetilde g}B_R(\widetilde p)>c(n)R^n,
 \qquad\text{for every }R>0.
\]
On the other hand, Theorem~\ref{thm:main} applied to
$(\widetilde M,\widetilde g)$, gives constants $C=C(n,m)$ and
$\nu=\nu(n,m)>0$ such that
\[
 \sup_{\widetilde p\in\widetilde M}
 \Vol_{\widetilde g}B_R(\widetilde p)
 \leq CR^m
 \qquad\text{whenever }\varepsilon R\leq\nu.
\]
Choose $L=L(n,m)$ sufficiently large that $CL^m<c(n)L^n$, and then
choose $\varepsilon=\varepsilon(n,m)>0$ sufficiently small so that
$\varepsilon L\leq\nu$. Hence, taking $R=L$ contradicts the two
previous inequalities.  

Thus $M$ is not rationally essential.  Finally,
an aspherical $M$ has $c_M$ as a homotopy equivalence, hence $M$ is not aspherical as well; and Gromov's
mapping theorem \cite[Section~3.1, Corollary~(B)]{GromovVolumeBounded} gives
$\lVert M\rVert=\lVert(c_M)_*[M]_{\mathbb R}\rVert_1=0$.
\end{proof}

\bibliographystyle{amsplain}
\bibliography{PaperVolumeGrowth}

@article{AndersonTopology,
  author  = {Anderson, Michael T.},
  title   = {On the topology of complete manifolds of non-negative {Ricci} curvature},
  journal = {Topology},
  volume  = {29},
  number  = {1},
  pages   = {41--55},
  year    = {1990},
  doi     = {10.1016/0040-9383(90)90024-E},
  url     = {https://doi.org/10.1016/0040-9383(90)90024-E},
  note    = {\href{https://doi.org/10.1016/0040-9383(90)90024-E}{doi:10.1016/0040-9383(90)90024-E}}
}

@unpublished{AntonelliXu,
  author        = {Antonelli, Gioacchino and Xu, Kai},
  title         = {New spectral {Bishop--Gromov} and {Bonnet--Myers} theorems and applications to isoperimetry},
  year          = {2024},
  eprint        = {2405.08918},
  archivePrefix = {arXiv},
  primaryClass  = {math.DG},
  url           = {https://arxiv.org/abs/2405.08918},
  note          = {Accepted for publication in \emph{J. Eur. Math. Soc. (JEMS)}; \href{https://arxiv.org/abs/2405.08918}{arXiv:2405.08918}}
}

@article{BrendleHirschJohne,
  author  = {Brendle, Simon and Hirsch, Sven and Johne, Florian},
  title   = {A generalization of {Geroch}'s conjecture},
  journal = {Comm. Pure Appl. Math.},
  volume  = {77},
  number  = {1},
  pages   = {441--456},
  year    = {2024},
  doi     = {10.1002/cpa.22137},
  url     = {https://doi.org/10.1002/cpa.22137},
  note    = {\href{https://doi.org/10.1002/cpa.22137}{doi:10.1002/cpa.22137}}
}

@article{CheegerColdingWarped,
  author  = {Cheeger, Jeff and Colding, Tobias H.},
  title   = {Lower bounds on {Ricci} curvature and the almost rigidity of warped products},
  journal = {Ann. of Math. (2)},
  volume  = {144},
  number  = {1},
  pages   = {189--237},
  year    = {1996},
  doi     = {10.2307/2118589},
  url     = {https://doi.org/10.2307/2118589},
  note    = {\href{https://doi.org/10.2307/2118589}{doi:10.2307/2118589}}
}

@article{CheegerColdingI,
  author  = {Cheeger, Jeff and Colding, Tobias H.},
  title   = {On the structure of spaces with {Ricci} curvature bounded below. {I}},
  journal = {J. Differential Geom.},
  volume  = {46},
  number  = {3},
  pages   = {406--480},
  year    = {1997},
  doi     = {10.4310/jdg/1214459974},
  url     = {https://doi.org/10.4310/jdg/1214459974},
  note    = {\href{https://doi.org/10.4310/jdg/1214459974}{doi:10.4310/jdg/1214459974}}
}

@article{CheegerNaber,
  author  = {Cheeger, Jeff and Naber, Aaron},
  title   = {Regularity of {Einstein} manifolds and the codimension $4$ conjecture},
  journal = {Ann. of Math. (2)},
  volume  = {182},
  number  = {3},
  pages   = {1093--1165},
  year    = {2015},
  doi     = {10.4007/annals.2015.182.3.5},
  url     = {https://doi.org/10.4007/annals.2015.182.3.5},
  note    = {\href{https://doi.org/10.4007/annals.2015.182.3.5}{doi:10.4007/annals.2015.182.3.5}}
}

@article{ChodoshLiStryker,
  author  = {Chodosh, Otis and Li, Chao and Stryker, Douglas},
  title   = {Volume growth of $3$-manifolds with scalar curvature lower bounds},
  journal = {Proc. Amer. Math. Soc.},
  volume  = {151},
  number  = {10},
  pages   = {4501--4511},
  year    = {2023},
  doi     = {10.1090/proc/16521},
  url     = {https://doi.org/10.1090/proc/16521},
  note    = {\href{https://doi.org/10.1090/proc/16521}{doi:10.1090/proc/16521}}
}

@article{CucinottaMondino,
  author  = {Cucinotta, Alessandro and Mondino, Andrea},
  title   = {On manifolds with almost non-negative {Ricci} curvature and integrally-positive {$k^{\mathrm{th}}$}-scalar curvature},
  journal = {Math. Ann.},
  volume  = {394},
  number  = {2},
  pages   = {Paper No.~49},
  year    = {2026},
  doi     = {10.1007/s00208-026-03406-8},
  url     = {https://doi.org/10.1007/s00208-026-03406-8},
  note    = {\href{https://doi.org/10.1007/s00208-026-03406-8}{doi:10.1007/s00208-026-03406-8}}
}

@article{GigliSplitting,
  author        = {Gigli, Nicola},
  title         = {The splitting theorem in non-smooth context},
  journal       = {Mem. Amer. Math. Soc.},
  volume        = {317},
  number        = {1609},
  pages         = {117 pp.},
  year          = {2026},
  doi           = {10.1090/memo/1609},
  eprint        = {1302.5555},
  archivePrefix = {arXiv},
  primaryClass  = {math.MG},
  url           = {https://doi.org/10.1090/memo/1609},
  note          = {\href{https://doi.org/10.1090/memo/1609}{doi:10.1090/memo/1609}; \href{https://arxiv.org/abs/1302.5555}{arXiv:1302.5555}}
}

@inproceedings{GromovLarge,
  author    = {Gromov, Mikhail},
  title     = {Large {Riemannian} manifolds},
  booktitle = {Curvature and Topology of Riemannian Manifolds (Katata, 1985)},
  editor    = {Shiohama, Katsuhiro and Sakai, Takashi and Sunada, Toshikazu},
  series    = {Lecture Notes in Mathematics},
  volume    = {1201},
  pages     = {108--121},
  publisher = {Springer-Verlag},
  address   = {Berlin},
  year      = {1986},
  doi       = {10.1007/BFb0075649},
  url       = {https://doi.org/10.1007/BFb0075649},
  note      = {\href{https://doi.org/10.1007/BFb0075649}{doi:10.1007/BFb0075649}}
}

@unpublished{HuangHuangSlowGrowth,
  author        = {Huang, Hongzhi and Huang, Xian-Tao},
  title         = {Splitting and slow volume growth for open manifolds with nonnegative {Ricci} curvature},
  year          = {2025},
  eprint        = {2510.22708},
  archivePrefix = {arXiv},
  primaryClass  = {math.DG},
  url           = {https://arxiv.org/abs/2510.22708v2},
  note          = {Revised preprint; \href{https://arxiv.org/abs/2510.22708v2}{arXiv:2510.22708v2}}
}

@article{HuangLiu,
  author  = {Huang, Xian-Tao and Liu, Shuai},
  title   = {Optimal asymptotic volume ratio for noncompact $3$-manifolds with asymptotically nonnegative {Ricci} curvature and a uniformly positive scalar curvature lower bound},
  journal = {J. Geom. Anal.},
  volume  = {35},
  number  = {3},
  pages   = {Paper No.~78, 21 pp.},
  year    = {2025},
  doi     = {10.1007/s12220-025-01919-3},
  url     = {https://doi.org/10.1007/s12220-025-01919-3},
  note    = {\href{https://doi.org/10.1007/s12220-025-01919-3}{doi:10.1007/s12220-025-01919-3}}
}

@phdthesis{Jansen,
  author        = {Jansen, Dorothea},
  title         = {Existence of typical scales for manifolds with lower {Ricci} curvature bound},
  school        = {Westf\"alische Wilhelms-Universit\"at M\"unster},
  year          = {2016},
  eprint        = {1703.09590},
  archivePrefix = {arXiv},
  primaryClass  = {math.DG},
  url           = {https://arxiv.org/abs/1703.09590},
  note          = { \href{https://arxiv.org/abs/1703.09590}{arXiv:1703.09590}}
}

@unpublished{KapovitchWilking,
  author        = {Kapovitch, Vitali and Wilking, Burkhard},
  title         = {Structure of fundamental groups of manifolds with {Ricci} curvature bounded below},
  year          = {2011},
  eprint        = {1105.5955},
  archivePrefix = {arXiv},
  primaryClass  = {math.DG},
  url           = {https://arxiv.org/abs/1105.5955v2},
  note          = {Revised preprint; \href{https://arxiv.org/abs/1105.5955v2}{arXiv:1105.5955v2}}
}

@article{Labbi,
  author        = {Labbi, Mohammed-Larbi},
  title         = {On {Weitzenb\"ock} curvature operators},
  journal       = {Math. Nachr.},
  volume        = {288},
  number        = {4},
  pages         = {402--411},
  year          = {2015},
  doi           = {10.1002/mana.201300352},
  eprint        = {math/0607521},
  archivePrefix = {arXiv},
  primaryClass  = {math.DG},
  url           = {https://doi.org/10.1002/mana.201300352},
  note          = {\href{https://doi.org/10.1002/mana.201300352}{doi:10.1002/mana.201300352}; \href{https://arxiv.org/abs/math/0607521}{arXiv:math/0607521}}
}

@article{GuthLargeBalls,
  author        = {Guth, Larry},
  title         = {Volumes of Balls in Large {R}iemannian Manifolds},
  journal       = {Ann. of Math. (2)},
  volume        = {173},
  number        = {1},
  pages         = {51--76},
  year          = {2011},
  doi           = {10.4007/annals.2011.173.1.2},
  eprint        = {math/0610212},
  archivePrefix = {arXiv},
  primaryClass  = {math.DG},
  mrnumber      = {2753599},
  url           = {https://doi.org/10.4007/annals.2011.173.1.2},
  note          = {\href{https://doi.org/10.4007/annals.2011.173.1.2}
                   {doi:10.4007/annals.2011.173.1.2};
                   \href{https://arxiv.org/abs/math/0610212}
                   {arXiv:math/0610212}}
}

@article{GuthUrysohnWidth,
  author        = {Guth, Larry},
  title         = {Volumes of Balls in {R}iemannian Manifolds and
                   {Uryson} Width},
  journal       = {J. Topol. Anal.},
  volume        = {9},
  number        = {2},
  pages         = {195--219},
  year          = {2017},
  doi           = {10.1142/S1793525317500029},
  eprint        = {1504.07886},
  archivePrefix = {arXiv},
  primaryClass  = {math.DG},
  mrnumber      = {3622232},
  url           = {https://doi.org/10.1142/S1793525317500029},
  note          = {\href{https://doi.org/10.1142/S1793525317500029}
                   {doi:10.1142/S1793525317500029};
                   \href{https://arxiv.org/abs/1504.07886}
                   {arXiv:1504.07886}}
}

@article{NabutovskyWidth,
  author        = {Nabutovsky, Alexander},
  title         = {Linear Bounds for Constants in {Gromov}'s Systolic
                   Inequality and Related Results},
  journal       = {Geom. Topol.},
  volume        = {26},
  number        = {7},
  pages         = {3123--3142},
  year          = {2022},
  doi           = {10.2140/gt.2022.26.3123},
  eprint        = {1909.12225},
  archivePrefix = {arXiv},
  primaryClass  = {math.MG},
  mrnumber      = {4540902},
  url           = {https://doi.org/10.2140/gt.2022.26.3123},
  note          = {\href{https://doi.org/10.2140/gt.2022.26.3123}
                   {doi:10.2140/gt.2022.26.3123};
                   \href{https://arxiv.org/abs/1909.12225}
                   {arXiv:1909.12225}}
}

@misc{Koirala2026,
  author        = {Koirala, Robert},
  title         = {Volume Growth under Positive Intermediate Curvature
                   and a {R}icci Lower Bound},
  year          = {2026},
  eprint        = {2608.17977},
  archivePrefix = {arXiv},
  primaryClass  = {math.DG},
  doi           = {10.48550/arXiv.2608.17977},
  url           = {https://arxiv.org/abs/2608.17977},
  note          = {\href{https://doi.org/10.48550/arXiv.2608.17977}
                   {doi:10.48550/arXiv.2608.17977};
                   \href{https://arxiv.org/abs/2608.17977}
                   {arXiv:2608.17977}}
}

@article{MunteanuWangIntegral,
  author        = {Munteanu, Ovidiu and Wang, Jiaping},
  title         = {Sharp Integral Bound of Scalar Curvature on
                   {$3$}-Manifolds},
  journal       = {Trans. Amer. Math. Soc.},
  year          = {2026},
  doi           = {10.1090/tran/9744},
  eprint        = {2505.10520},
  archivePrefix = {arXiv},
  primaryClass  = {math.DG},
  url           = {https://doi.org/10.1090/tran/9744},
  note          = {\href{https://doi.org/10.1090/tran/9744}
                   {doi:10.1090/tran/9744};
                   \href{https://arxiv.org/abs/2505.10520}
                   {arXiv:2505.10520}}
}

@article{ChodoshLiAspherical,
  author        = {Chodosh, Otis and Li, Chao},
  title         = {Generalized Soap Bubbles and the Topology of Manifolds
                   with Positive Scalar Curvature},
  journal       = {Ann. of Math. (2)},
  volume        = {199},
  number        = {2},
  pages         = {707--740},
  year          = {2024},
  doi           = {10.4007/annals.2024.199.2.3},
  mrnumber      = {4713021},
  eprint        = {2008.11888},
  archivePrefix = {arXiv},
  primaryClass  = {math.DG},
  url           = {https://doi.org/10.4007/annals.2024.199.2.3},
  note          = {\href{https://doi.org/10.4007/annals.2024.199.2.3}
                   {doi:10.4007/annals.2024.199.2.3};
                   \href{https://arxiv.org/abs/2008.11888}
                   {arXiv:2008.11888}}
}

@unpublished{GromovAspherical5,
  author        = {Gromov, Mikhail},
  title         = {No Metrics with Positive Scalar Curvatures on
                   Aspherical 5-Manifolds},
  year          = {2020},
  eprint        = {2009.05332},
  archivePrefix = {arXiv},
  primaryClass  = {math.DG},
  doi           = {10.48550/arXiv.2009.05332},
  url           = {https://arxiv.org/abs/2009.05332},
  note          = {\href{https://arxiv.org/abs/2009.05332}
                   {arXiv:2009.05332}}
}

@incollection{LiokumovichMaximo,
  author        = {Liokumovich, Yevgeny and Maximo, Davi},
  title         = {Waist inequality for $3$-manifolds with positive scalar curvature},
  booktitle     = {Perspectives in Scalar Curvature},
  editor        = {Gromov, Mikhail L. and Lawson, Jr., H. Blaine},
  volume        = {2},
  pages         = {799--831},
  publisher     = {World Scientific Publishing},
  address       = {Hackensack, NJ},
  year          = {2023},
  doi           = {10.1142/9789811273230_0022},
  eprint        = {2012.12478},
  archivePrefix = {arXiv},
  primaryClass  = {math.DG},
  url           = {https://doi.org/10.1142/9789811273230_0022},
  note          = {MR~4577931; \href{https://doi.org/10.1142/9789811273230_0022}{doi:10.1142/9789811273230\_0022}; \href{https://arxiv.org/abs/2012.12478}{arXiv:2012.12478}}
}

@article{LiokumovichWang,
  author        = {Liokumovich, Yevgeny and Wang, Zhichao},
  title         = {On the waist and width inequality in complete $3$-manifolds with positive scalar curvature},
  journal       = {Cambridge J. Math.},
  volume        = {14},
  number        = {2},
  pages         = {349--375},
  year          = {2026},
  doi           = {10.4310/CJM.260514230515},
  eprint        = {2308.04044},
  archivePrefix = {arXiv},
  primaryClass  = {math.DG},
  url           = {https://doi.org/10.4310/CJM.260514230515},
  note          = {\href{https://doi.org/10.4310/CJM.260514230515}{doi:10.4310/CJM.260514230515}; \href{https://arxiv.org/abs/2308.04044}{arXiv:2308.04044}}
}

@article{MunteanuWang,
  author  = {Munteanu, Ovidiu and Wang, Jiaping},
  title   = {Geometry of three-dimensional manifolds with positive scalar curvature},
  journal = {Amer. J. Math.},
  volume  = {148},
  number  = {1},
  pages   = {131--160},
  year    = {2026},
  doi     = {10.1353/ajm.2026.a980770},
  url     = {https://doi.org/10.1353/ajm.2026.a980770},
  note    = {\href{https://doi.org/10.1353/ajm.2026.a980770}{doi:10.1353/ajm.2026.a980770}}
}

@article{Papasoglu,
  author  = {Papasoglu, Panos},
  title   = {Uryson width and volume},
  journal = {Geom. Funct. Anal.},
  volume  = {30},
  number  = {2},
  pages   = {574--587},
  year    = {2020},
  doi     = {10.1007/s00039-020-00533-5},
  url     = {https://doi.org/10.1007/s00039-020-00533-5},
  note    = {\href{https://doi.org/10.1007/s00039-020-00533-5}{doi:10.1007/s00039-020-00533-5}}
}

@article{Petrunin,
  author  = {Petrunin, Anton M.},
  title   = {An upper bound for the curvature integral},
  journal = {St. Petersburg Math. J.},
  volume  = {20},
  number  = {2},
  pages   = {255--265},
  year    = {2009},
  doi     = {10.1090/S1061-0022-09-01046-2},
  url     = {https://doi.org/10.1090/S1061-0022-09-01046-2},
  note    = {\href{https://doi.org/10.1090/S1061-0022-09-01046-2}{doi:10.1090/S1061-0022-09-01046-2}}
}

@article{ShenYe,
  author  = {Shen, Ying and Ye, Rugang},
  title   = {On stable minimal surfaces in manifolds of positive bi-{Ricci} curvatures},
  journal = {Duke Math. J.},
  volume  = {85},
  number  = {1},
  pages   = {109--116},
  year    = {1996},
  doi     = {10.1215/S0012-7094-96-08505-1},
  url     = {https://doi.org/10.1215/S0012-7094-96-08505-1},
  note    = {\href{https://doi.org/10.1215/S0012-7094-96-08505-1}{doi:10.1215/S0012-7094-96-08505-1}}
}

@article{YWang,
  author  = {Wang, Yipeng},
  title   = {Effective volume growth of three-manifolds with positive scalar curvature},
  journal = {Proc. Amer. Math. Soc.},
  volume  = {154},
  number  = {1},
  pages   = {329--337},
  year    = {2026},
  doi     = {10.1090/proc/17082},
  url     = {https://doi.org/10.1090/proc/17082},
  note    = {\href{https://doi.org/10.1090/proc/17082}{doi:10.1090/proc/17082}}
}

@article{WangXieZhuZhu,
  author  = {Wang, Jinmin and Xie, Zhizhang and Zhu, Bo and Zhu, Xingyu},
  title   = {Positive scalar curvature meets {Ricci} limit spaces},
  journal = {Manuscripta Math.},
  volume  = {175},
  number  = {3--4},
  pages   = {943--969},
  year    = {2024},
  doi     = {10.1007/s00229-024-01596-6},
  url     = {https://doi.org/10.1007/s00229-024-01596-6},
  note    = {\href{https://doi.org/10.1007/s00229-024-01596-6}{doi:10.1007/s00229-024-01596-6}}
}

@article{WeiXuZhang,
  author  = {Wei, Guodong and Xu, Guoyi and Zhang, Shuai},
  title   = {Volume growth and positive scalar curvature},
  journal = {Trans. Amer. Math. Soc.},
  volume  = {378},
  number  = {9},
  pages   = {6109--6136},
  year    = {2025},
  doi     = {10.1090/tran/9280},
  url     = {https://doi.org/10.1090/tran/9280},
  note    = {\href{https://doi.org/10.1090/tran/9280}{doi:10.1090/tran/9280}}
}

@incollection{YauProblems,
  author    = {Yau, Shing-Tung},
  title     = {Open problems in geometry},
  booktitle = {{S. S. Chern}: A Great Geometer of the Twentieth Century},
  editor    = {Yau, Shing-Tung},
  series    = {Monographs in Geometry and Topology},
  publisher = {International Press},
  address   = {Hong Kong},
  pages     = {275--319},
  year      = {1992},
  isbn      = {962-7670-02-2},
  note      = {See Section~I, p.~278, Problem~9}
}

@article{ZhouZhu,
  author  = {Zhou, Jie and Zhu, Jintian},
  title   = {Optimal volume bound and volume growth for {Ricci}-nonnegative manifolds with positive bi-{Ricci} curvature},
  journal = {J. Reine Angew. Math.},
  volume  = {821},
  pages   = {1--21},
  year    = {2025},
  doi     = {10.1515/crelle-2024-0100},
  url     = {https://doi.org/10.1515/crelle-2024-0100},
  note    = {\href{https://doi.org/10.1515/crelle-2024-0100}{doi:10.1515/crelle-2024-0100}}
}

@article{BZhu,
  author  = {Zhu, Bo},
  title   = {Geometry of positive scalar curvature on complete manifold},
  journal = {J. Reine Angew. Math.},
  volume  = {791},
  pages   = {225--246},
  year    = {2022},
  doi     = {10.1515/crelle-2022-0049},
  url     = {https://doi.org/10.1515/crelle-2022-0049},
  note    = {\href{https://doi.org/10.1515/crelle-2022-0049}{doi:10.1515/crelle-2022-0049}}
}

@article{BZhuXZhu,
  author  = {Zhu, Bo and Zhu, Xingyu},
  title   = {Optimal diameter estimate of three-dimensional {Ricci} limit spaces},
  journal = {Proc. Amer. Math. Soc.},
  volume  = {152},
  number  = {2},
  pages   = {815--821},
  year    = {2024},
  doi     = {10.1090/proc/16529},
  url     = {https://doi.org/10.1090/proc/16529},
  note    = {\href{https://doi.org/10.1090/proc/16529}{doi:10.1090/proc/16529}}
}

@article{XZhu,
  author  = {Zhu, Xingyu},
  title   = {Two-dimension vanishing, splitting and positive scalar curvature},
  journal = {Math. Ann.},
  volume  = {392},
  number  = {2},
  pages   = {2635--2656},
  year    = {2025},
  doi     = {10.1007/s00208-025-03139-0},
  url     = {https://doi.org/10.1007/s00208-025-03139-0},
  note    = {\href{https://doi.org/10.1007/s00208-025-03139-0}{doi:10.1007/s00208-025-03139-0}}
}

@article{BraunSauer2021,
  author        = {Braun, Sabine and Sauer, Roman},
  title         = {Volume and macroscopic scalar curvature},
  journal       = {Geom. Funct. Anal.},
  volume        = {31},
  number        = {6},
  pages         = {1321--1376},
  year          = {2021},
  doi           = {10.1007/s00039-021-00588-y},
  eprint        = {2012.08999},
  archivePrefix = {arXiv},
  primaryClass  = {math.DG},
  url           = {https://doi.org/10.1007/s00039-021-00588-y},
  note          = {\href{https://doi.org/10.1007/s00039-021-00588-y}{doi:10.1007/s00039-021-00588-y}; \href{https://arxiv.org/abs/2012.08999}{arXiv:2012.08999}}
}

@article{CheegerGromoll,
  author  = {Cheeger, Jeff and Gromoll, Detlef},
  title   = {The splitting theorem for manifolds of nonnegative
             {R}icci curvature},
  journal = {J. Differential Geom.},
  volume  = {6},
  number  = {1},
  pages   = {119--128},
  year    = {1971},
  doi     = {10.4310/jdg/1214430220},
  url     = {https://doi.org/10.4310/jdg/1214430220},
  note    = {\href{https://doi.org/10.4310/jdg/1214430220}{doi:10.4310/jdg/1214430220}}
}

@unpublished{GeGromov,
  author        = {Ge, Jian},
  title         = {Heat kernel geometry and {Gromov}'s volume growth conjecture},
  year          = {2026},
  eprint        = {2608.13553},
  archivePrefix = {arXiv},
  primaryClass  = {math.DG},
  doi           = {10.48550/arXiv.2608.13553},
  url           = {https://arxiv.org/abs/2608.13553},
  note          = {\href{https://arxiv.org/abs/2608.13553}{arXiv:2608.13553}}
}

@article{GromovVolumeBounded,
  author  = {Gromov, Mikhail},
  title   = {Volume and bounded cohomology},
  journal = {Publ. Math. Inst. Hautes \'Etudes Sci.},
  volume  = {56},
  pages   = {5--99},
  year    = {1982},
  url     = {https://www.numdam.org/item/PMIHES_1982__56__5_0/},
  note    = {\href{https://www.numdam.org/item/PMIHES_1982__56__5_0/}{Numdam}}
}

@inproceedings{GuthMacroscopic,
  author        = {Guth, Larry},
  title         = {Metaphors in systolic geometry},
  booktitle     = {Proceedings of the International Congress of
                   Mathematicians. Volume II},
  pages         = {745--768},
  publisher     = {Hindustan Book Agency},
  address       = {New Delhi},
  year          = {2010},
  doi           = {10.1142/9789814324359_0072},
  eprint        = {1003.4247},
  archivePrefix = {arXiv},
  primaryClass  = {math.DG},
  mrnumber      = {2827817},
  url           = {https://doi.org/10.1142/9789814324359_0072},
  note          = {\href{https://doi.org/10.1142/9789814324359_0072}{doi:10.1142/9789814324359\_0072}; \href{https://arxiv.org/abs/1003.4247}{arXiv:1003.4247}}
}

@misc{KZ,
  author        = {Kong, Bochao and Zhu, Xingyu},
  title         = {Positive Scalar Curvature and Volume Growth},
  year          = {2026},
  eprint        = {2608.14438},
  archivePrefix = {arXiv},
  primaryClass  = {math.DG},
  doi           = {10.48550/arXiv.2608.14438},
  url           = {https://arxiv.org/abs/2608.14438},
  note          = {\href{https://arxiv.org/abs/2608.14438}{arXiv:2608.14438}}
}

@unpublished{KumarSen,
  author        = {Kumar, Aditya and Sen, Balarka},
  title         = {{Urysohn} width and macroscopic scalar curvature},
  year          = {2026},
  eprint        = {2601.14669},
  archivePrefix = {arXiv},
  primaryClass  = {math.DG},
  doi           = {10.48550/arXiv.2601.14669},
  url           = {https://arxiv.org/abs/2601.14669},
  note          = {\href{https://arxiv.org/abs/2601.14669}{arXiv:2601.14669}}
}

@article{Lott2025,
  author        = {Lott, John},
  title         = {Some obstructions to positive scalar curvature on a noncompact manifold},
  journal       = {J. Reine Angew. Math.},
  volume        = {829},
  pages         = {247--283},
  year          = {2025},
  doi           = {10.1515/crelle-2025-0072},
  eprint        = {2402.13239},
  archivePrefix = {arXiv},
  primaryClass  = {math.DG},
  url           = {https://doi.org/10.1515/crelle-2025-0072},
  note          = {\href{https://doi.org/10.1515/crelle-2025-0072}{doi:10.1515/crelle-2025-0072}; \href{https://arxiv.org/abs/2402.13239}{arXiv:2402.13239}}
}

@article{LiokumovichLishakNabutovskyRotman,
  author        = {Liokumovich, Yevgeny and Lishak, Boris and
                   Nabutovsky, Alexander and Rotman, Regina},
  title         = {Filling metric spaces},
  journal       = {Duke Math. J.},
  volume        = {171},
  number        = {3},
  pages         = {595--632},
  year          = {2022},
  doi           = {10.1215/00127094-2021-0039},
  eprint        = {1905.06522},
  archivePrefix = {arXiv},
  primaryClass  = {math.DG},
  url           = {https://doi.org/10.1215/00127094-2021-0039},
  note          = {\href{https://doi.org/10.1215/00127094-2021-0039}{doi:10.1215/00127094-2021-0039}; \href{https://arxiv.org/abs/1905.06522}{arXiv:1905.06522}}
}

@incollection{GromovFourLectures,
  author        = {Gromov, Mikhail},
  title         = {Four Lectures on Scalar Curvature},
  booktitle     = {Perspectives in Scalar Curvature},
  editor        = {Gromov, Mikhail L. and Lawson, Jr., H. Blaine},
  volume        = {1},
  pages         = {1--514},
  publisher     = {World Scientific Publishing Co. Pte. Ltd.},
  address       = {Hackensack, NJ},
  year          = {2023},
  doi           = {10.1142/9789811273223_0001},
  eprint        = {1908.10612},
  archivePrefix = {arXiv},
  primaryClass  = {math.DG},
  mrnumber      = {4577903},
  url           = {https://doi.org/10.1142/9789811273223_0001},
  note          = {\href{https://doi.org/10.1142/9789811273223_0001}{doi:10.1142/9789811273223\_0001}; \href{https://arxiv.org/abs/1908.10612}{arXiv:1908.10612}}
}

\end{document}